\documentclass[a4paper,11pt,oneside]{article}

\makeindex
\usepackage{hyperref}
\usepackage{graphicx}
\usepackage{enumerate}
\hypersetup{colorlinks, linkcolor=blue, urlcolor=red, filecolor=green, citecolor=blue}
\usepackage{amsmath}
\usepackage{amsthm}
\usepackage{amssymb}
\usepackage{authblk}
\usepackage{enumerate}
\usepackage{fullpage}
\usepackage{esint}

\theoremstyle{plain}
\newtheorem{theorem}{Theorem}[section]
\newtheorem{corollary}[theorem]{Corollary}
\newtheorem{lemma}[theorem]{Lemma}
\newtheorem{proposition}[theorem]{Proposition}

\theoremstyle{definition}
\newtheorem{definition}[theorem]{Definition}

\newtheorem{notation}[theorem]{Notation}

\theoremstyle{remark}
\newtheorem{remark}{Remark}

\title{Generalized transport inequalities and concentration bounds for Riesz-type gases}
\author{David García-Zelada, David Padilla-Garza}

\begin{document}

\maketitle

\begin{abstract}
    This paper explores the connection between a generalized Riesz electric energy and norms on the set of probability measures defined in terms of duality. We derive functional inequalities linking these two notions, recovering and generalizing existing Coulomb transport inequalities. We then use them to prove concentration of measure around the equilibrium and thermal equilibrium measures. Finally, we leverage these concentration inequalities to obtain Moser-Trudinger-type inequalities, which may also be interpreted as bounds on the Laplace transform of fluctuations. 
\end{abstract}

\section{Introduction and motivation}

We will study a many-particle system, in which the particles interact via a repulsive kernel, and are confined by an external potential. This is modelled by the Hamiltonian
\begin{equation}
\label{eq:hamiltonian}
    \mathcal{H}_{N}(X_{N}) = \sum_{i \neq j} g(x_{i} - x_{j}) + N \sum_{i=1}^{N} V(x_{i}),
\end{equation}
where $g: \mathbf{R}^{d} \to \mathbf{R}$ is the repulsive kernel, $V: \mathbf{R}^{d} \to \mathbf{R}$ is the confining potential, $N$ is the number of particles, and $X_{N} = (x_{1}, ... , x_{N}) \in \mathbf{R}^{d \times N}$. We are interested in the behaviour for large but finite $N$. A system modeled by equation \eqref{eq:hamiltonian} will be called an interacting gas, or many-particle system. \\

At zero temperature, the particles will arrange themselves into the configuration that minimizes the Hamiltonian. However, we may also study many-particle systems at positive temperature. In this case the position of the particles is a random variable in $\mathbf{R}^{d \times N}$. The density of the random variable is given by the Gibbs measure,
\begin{equation}
   \mathrm d \mathbf{P}_{N, \beta} (X_{N}) = \frac{1}{Z_{N, \beta}} \exp \left( - \beta \mathcal{H}_{N}(X_{N}) \right) \, \mathrm d X_{N},
\end{equation}
where
\begin{equation}
\label{eq:partfunc}
    Z_{N, \beta} = \int_{\mathbf{R}^{d \times N}} \exp \left( - \beta \mathcal{H}_{N}(X_{N}) \right) \, \mathrm d X_{N}
\end{equation}
is the partition function, and $\beta >0$ is the inverse temperature which may depend on $N$. \\

The most frequent form of the repulsive interaction $g$ is given by
\begin{equation}
\label{Coulombcase}
\begin{cases}
g(x)=\frac{c_{d}}{|x|^{d-2}} \text{ if } d \geq 3, \\
g(x)=-c_{2} \log(|x|) \text{ if } d = 2,
\end{cases}
\end{equation}
where $c_{d}$ is such that
\begin{equation}
    -\Delta g  = \delta_{0}
\end{equation}
for all $d \geq 2$. We will refer to this setting as the Coulomb case. Another frequent form of $g$ is given by 
\begin{equation}\label{Rieszcase}
    g(x) = \frac{c_{d,s}}{|x|^{d-2s}},
\end{equation}
with $0<s<\min\{ \frac{d}{2},1 \}$,  $d \geq 1$, and $c_{d,s}$ such that
\begin{equation}
    (- \Delta)^{s} g = \delta_{0}.
\end{equation}
We will refer to this setting as the Riesz case. This paper will deal with a generalization of Riesz interactions, which are very similar to ones first introduced in \cite{nguyen2021mean, rosenzweig2021global}. \\

Apart from the Riesz case, it is also possible to study a many-particle system with interactions given by a hypersingular (non-integrable) interaction, i.e., $g$ given by equation \eqref{Rieszcase} with $s <0$, see for example \cite{hardin2017generating, hardin2018large, hardin2020dynamics}.\\ 

Coulomb and Riesz gases are a classical field with applications in spherical packing \cite{torquato2016hyperuniformity, cohn2007universally, cohn2017sphere, viazovska2017sphere, petrache2020crystallization}, statistical mechanics \cite{alastuey1981classical, jancovici1993large, sari1976nu, girvin2005introduction, stormer1999fractional}, random matrix theory \cite{johansson1998fluctuations, bourgade2012bulk, bourgade2014universality, bourgade2014local, lambert2021mesoscopic, hardy2021clt, lambert2019quantitative},and  mathematical physics \cite{berman2014determinantal, rougerie2015incompressibility}, among other fields. \\

The study of a general {interacting gas} has recently begun to attract attention \cite{garcia2019large, lambert2021poisson, chafai2014first, nguyen2021mean, rosenzweig2021global}.\\

\section{Main definitions}

In this section, we introduce the necessary objects and notation in order to state the main results of this paper.

\subsection{Interacting gases}

We start with notions related to interacting gases.

The most fundamental observable is the empirical measure.
\begin{definition}
The empirical measure ${\rm emp}_{N}$ is defined as 
\begin{equation}
 {\rm emp}_{N} (X_{N}) = \frac{1}{N} \sum_{i=1}^{N} \delta_{x_{i}}.   
\end{equation}
In order to ease notation, we will often write $ {\rm emp}_{N}$ instead of $ {\rm emp}_{N} (X_{N})$. 
\end{definition}

We proceed to define quantities related to the electric energy and entropy. 

\begin{definition}
We denote the electric self-interaction of a measure $\mu$ by $\mathcal{E}(\mu)$:
\begin{equation}
    \mathcal{E} (\mu) = \int_{\mathbf{R}^{d} \times \mathbf{R}^{d}} g(x-y) \mathrm d  \mu \otimes \mu(x,y).
\end{equation}

We denote the mean field limit of $\mathcal{H}_{N}$ by $\mathcal{E}_{V}$, acting on a measure $\mu$ by
\begin{equation}\label{meanfieldlimit}
    \mathcal{E}_{V} (\mu) =  \mathcal{E} (\mu) + \int_{\mathbf{R}^{d}} V 
    \mathrm d \mu. 
\end{equation}

We also introduce the free energy $ \mathcal{E}_{V}^{\theta}$, acting on a measure $\mu$ by
\begin{equation}\label{thermallimit}
     \mathcal{E}_{V}^{\theta} (\mu) =  \mathcal{E}_{V} (\mu) + \frac{1}{ \theta} {\rm ent}[\mu], 
\end{equation}
with
\begin{equation}
{\rm ent}[\mu ]=
    \begin{cases}
    \int_{\mathbf{R}^{d}}  \log  ({\mathrm d\mu/ \mathrm d \mathcal{L}})
    \, 
    \mathrm d{\mu} \quad  \mbox{ if }  \mu \ll \mathcal{L}, \  \log  ({\mathrm d\mu/ \mathrm d \mathcal{L}}) \in L^{1}(\mathbf{R}^{d}) \\
    \infty\quad  \text{ otherwise,}
    \end{cases}
\end{equation} 
where $\mathcal{L}$ denotes the Lebesgue measure on $\mathbf{R}^{d}$.

We define $\mathcal{E}^{\neq}$ for a measure $\mu$ as 
\begin{equation}
    \mathcal{E}^{\neq} (\mu) = \iint_{\mathbf{R}^{d} \times \mathbf{R}^{d} \setminus \Delta} g(x-y) \mathrm d \mu\otimes  \mu (x,y),
\end{equation}
where 
\begin{equation}
    \Delta = \{ (x,y) \in \mathbf{R}^{d} \times \mathbf{R}^{d} \}.
\end{equation}
We also define
\begin{equation}
    {\rm F}_{N}(X_{N}, \mu) = \mathcal{E}^{\neq} (\mu - {\rm emp}_{N}(X_{N})).
\end{equation}
The reason for excluding the diagonal in the integral is so that these quantities can be finite on atomic measures. 
Lastly, given a measure $\mu$ on $\mathbf{R}^{d}$, we define
\begin{equation}
    h^{\mu} = \mu \ast g.
\end{equation}

\end{definition}

We proceed to defining the equilibrium and thermal equilibrium measures, which will play a central role in the rest of the paper. 

\begin{notation}
Given a measurable set $\Omega \subset \mathbf{R}^{d}$, we denote by $\mathcal{P}(\Omega)$ the set of probability measures on $\Omega$.
\end{notation}

\begin{definition}
We denote by $\mu_{\infty}$ the minimizer of $ \mathcal{E}_{V}$ in $\mathcal{P}(\mathbf{R}^{d})$:
\begin{equation}
\label{def:eqmeas}
    \mu_{\infty}:= \underset{\mu \in \mathcal{P}(\mathbf{R}^{d})}{\rm argmin} \, \mathcal{E}_{V}(\mu).
\end{equation}
We will refer to $\mu_{\infty}$ as the equilibrium measure.

We denote by $\Sigma$ the support of $\mu_{\infty}$. 

We denote by $\mu_{\theta}$ the minimizer of $ \mathcal{E}_{V}^{\theta}$ in $\mathcal{P}(\mathbf{R}^{d})$:
\begin{equation}
\label{def:theqmeas}
    \mu_{\theta}:= \underset{\mu \in \mathcal{P}(\mathbf{R}^{d})}{\rm argmin}\,  \mathcal{E}_{V}^{\theta} (\mu).
\end{equation}
We will refer to $\mu_{\theta}$ as the thermal equilibrium measure.

We refer to Section~\ref{preliminaries} for existence, uniqueness and basic properties of $\mu_{\infty}$ and $\mu_{\theta}$.
\end{definition}

\begin{remark}
As long as the temperature is not too big ($\frac{1}{N} \ll \beta$), the equilibrium measure provides a good approximation to the empirical measure. A better approximation, however, is provided by $\mu_{N \beta}$, the thermal equilibrium measure with parameter $N \beta$. This difference becomes bigger as the temperature becomes bigger. In the endpoint case $\beta = \frac{1}{N}$, $\mu_{N \beta}$ does not converge to $\mu_{\infty}$ as $N$ tends to $\infty$. In this case, $\mu_{N \beta}$ still provides a good approximation to ${\rm emp}_{N}$, but $\mu_{\infty}$ does not. 
\end{remark}

\begin{remark}
For the rest of the paper, we commit the abuse of notation of not distinguishing between a measure and its density. 
\end{remark}

As mentioned before, this paper deals with general interactions that qualitatively behave like Riesz interactions. We now specify this class of interactions. 

\begin{notation}
We denote by $\mathfrak{F}(f)$, or alternatively by $\widehat{f}$ the Fourier transform of $f.$
\end{notation}

\begin{definition}
\label{def:riesztype}
Let $s \in (0,\min\{1, \frac{d}{2}\})$. A function $g: \mathbf{R}^{d} \to \mathbf{R}$ is called a Riesz-type kernel of order $s$ if there exists an integer $m \geq 0$, a function $G: \mathbf{R}^{d+m} \to \mathbf{R}$ and constants $C_{1}, C_{2}, C \geq 0$ depending only on $d,s,m,$ and $G$ such that
\begin{itemize}

    \item[1.] $G(x,0) = g(x)$ for every $(x,0) \in \mathbf{R}^{d+m}$.

    \item[2.] $G(X) = G(-X)$ for every $X \in \mathbf{R}^{d+m}\setminus\{0\}$.
 
    \item[3.] $\lim_{X \to 0} G(X)= \infty$.
  
    \item[4.] There exists $r_{0} >0 \text{ such that } \Delta G \leq 0 \text{ in } B(0, r_{0}) \subset \mathbf{R}^{d+m}$.
    
     \item[5.] $\ \left|  G (X) \right| \leq C \left( \frac{1}{|X|^{d-2s}} \right)$ for every $X \in \mathbf{R}^{d+m}\setminus\{0\}$.
    
     \item[6.] $\ \left| \nabla G (X) \right| \leq C \left( \frac{1}{|X|^{d-2s +1}} \right)$ for every $X \in \mathbf{R}^{d+m}\setminus\{0\}$.

     \item[7.] $\frac{C_{1}}{\left| \Xi \right|^{m+2s}} \leq \widehat{G} (\Xi) \leq   \frac{C_{2}}{\left| \Xi \right|^{m+2s}}$.
     
     \item[8.] $\frac{C_{1}}{\left| \xi \right|^{2s}} \leq \widehat{g} (\xi) \leq   \frac{C_{2}}{\left| \xi \right|^{2s}}$.
     
     \item[9.] There exist $c_{s}<1$ and $ r_{0}>0 $ 
     such that, for every $X,Y \in B(0,r_{0}) \setminus\{0\} \text{ with } |Y| \geq 2 |X|$,
     the inequality 
     $G(Y) < c_{s} G(X)$ is satisfied.
     
     \item[10.] The function $h: \mathbf{R}^{d} \to \mathbf{R}$, defined by $h = \mathfrak{F}\left( {\frac{1}{\widehat{g}}} \right)$, satisfies $
         \frac{C_{1}}{\left| x \right|^{d+2s}} \leq h (x) \leq   \frac{C_{2}}{\left| x \right|^{d+2s}}$.
\end{itemize}
\end{definition}

\begin{remark}
Items 1-9 are basically found in $\cite{nguyen2021mean, rosenzweig2021global}$, but unlike $\cite{nguyen2021mean, rosenzweig2021global}$, we only impose a growth condition on $G$ and its first derivative. Item 10 is not found in $\cite{nguyen2021mean, rosenzweig2021global}$, but it is necessary to derive a monotonicity property, see Section~\ref{sect:mutheta}. 
\end{remark}

\begin{remark}
The most important example of a Riesz-type kernel of order $s$ is, of course, the Riesz kernel \eqref{Rieszcase} of order $s$. In this case, the function $G(X)$ is given by $G(X) = \frac{c_{d,s}}{|X|^{d-2s}} $, which will satisfy Definition \ref{def:riesztype} for $m$ big enough.
\end{remark} 

Apart from a general interaction, we will deal with a potential which is general except for growth and regularity conditions. We now specify the exact class of potentials that we will deal with. 

\begin{definition}
\label{def:admissibleV}
We call a potential $V: \mathbf{R}^{d} \to \mathbf{R}$ \emph{admissible} if
\begin{itemize}
    \item[1.] $V \in C^{2}$. 
    
    \item[2.] $\lim_{x \to \infty} V(x) = \infty$. 
    
    \item[3.] $\forall \beta>0$, 
    \begin{equation}
         \quad \int_{\mathbf{R}^{d}} \exp (-\beta V (x))\, \mathrm d x < \infty. 
    \end{equation}
    
    \item[4.] $\mu_{\infty} \in L^{\infty}$.
    
    \item[5.] If $d=2$ and $g$ is the Coulomb kernel, then 
    $\lim_{x \to \infty} (V(x) - \log |x|) = \infty$.
   
\end{itemize}
\end{definition}

\subsection{Functional analysis}

In this subsection, we define some norms and spaces used in the main results. 

\subsubsection{Hölder norms}

We start by defining the Hölder norm and seminorm
as well as their duals. 

\begin{definition}
For a continuous function $f: \Omega \to \mathbf{R}$, we define the Hölder $\alpha$ seminorm as
\begin{equation}
    |f|_{\dot{C}^{0,\alpha}} = \sup_{x,y \in \Omega, x \neq y} \frac{f(x) - f(y)}{|x-y|^{\alpha}}.
\end{equation}
We denote by $\dot{C}^{0,\alpha}(\Omega)$ the space
\begin{equation}
    \dot{C}^{0,\alpha}(\Omega) = \{ f \in C(\Omega) :\, |f|_{\dot{C}^{0,\alpha}} < \infty \}.
\end{equation}

For a continuous function $f: \Omega \to \mathbf{R}$, we define the full Hölder $\alpha$ norm as
\begin{equation}
     \|f\|_{{C}^{0,\alpha}} =  |f|_{\dot{C}^{0,\alpha}} +  \|f\|_{\infty}, 
\end{equation}
where $\|f\|_{\infty}$ is defined for a continuous function $f$ as the supremum of $|f|$.

We now define the corresponding dual norms. For a measure $\mu$ on $\Omega$, we define the dual Hölder $\alpha$ norm as 
\begin{equation}
    \|\mu\|_{\dot{C}^{0,\alpha}_{*}} = \sup_{f \in \dot{C}^{0,\alpha}(\Omega)} \frac{\int_{\Omega} f \, \mathrm d \mu}{  |f|_{\dot{C}^{0,\alpha}}}. 
\end{equation}

For a measure $\mu$ on $\Omega$, we define the dual full Hölder $\alpha$ norm as 
\begin{equation}
    \|\mu\|_{{C}^{0,\alpha}_{*}} = \sup_{f \in {C}^{0,\alpha}(\Omega)} \frac{\int_{\Omega} f \, \mathrm d \mu}{  \|f\|_{{C}^{0,\alpha}}}. 
\end{equation}

We also define the $\dot{C}^{1,\alpha}$ seminorm of $f: \Omega \to \mathbf{R}$ as 
\begin{equation}
    |f|_{\dot{C}^{1,\alpha}} = |f|_{\dot{C}^{0,\alpha}} + \max_{i} |\partial_{i} f|_{\dot{C}^{0,\alpha}}.
\end{equation}
\end{definition}

\begin{remark}
Note that when $\alpha=1$, the dual Hölder $\alpha$ norm is the Wasserstein $1-$distance, and the dual full Hölder $\alpha$ norm is the bounded-Lipschitz norm. 
\end{remark}

\begin{remark}
Note that if $\mu(\Omega) \neq 0$, then $ \|\mu\|_{\dot{C}^{0,\alpha}_{*}} = \infty$.
\end{remark}

\begin{remark} 
Whenever $\Omega$ is bounded, the norms
$\|\cdot\|_{{C}^{0,\alpha}_{*}}$ and
$\|\cdot\|_{\dot{C}^{0,\alpha}_{*}}$
are equivalent
on the set of measures of total mass zero.
\end{remark}

\subsubsection{On $H^{s}$ norms}

In this subsubsection, we deal with the $L^{2}$-based norm $H^{s}$. 

We begin with the definition of this norm. 
\begin{definition}
Given a tempered distribution $f$ on $\mathbf{R}^{d}$ 
and $s \in (0,1)$, we define the (Fourier) $\dot{H}^{s}$ seminorm of $f$ as 
\begin{equation}
    |f|_{\dot{H}_{F}^{s}}^{2} =  \int_{\mathbf{R}^{d}} |\widehat{f}(\xi)|^{2}
    |\xi|^{2s} \mathrm d \xi.
\end{equation}

We also define the (difference quotient) $\dot{H}^{s}$ seminorm, for $s \in (0,1)$, as 
\begin{equation}
      |f|_{\dot{H}_{dq}^{s}}^{2} = \iint_{\mathbf{R}^{d} \times \mathbf{R}^{d}} \frac{|f(x) - f(y)|^{2}}{|x-y|^{d + 2s}} \, \mathrm d x \mathrm d y.
\end{equation}
\end{definition}

These two quantities are equivalent, and this is the content of the next lemma.

\begin{lemma}
\label{lem:eq}
The norms $ | \cdot |_{\dot{H}_{dq}^{s}}$ and $| \cdot |_{\dot{H}_{F}^{s}}$ are equivalent.
\end{lemma}

\begin{proof}
See \cite{bahouri2011fourier}, chapter 1. 
\end{proof}

Motivated by Lemma \ref{lem:eq}, we will write  $ | \cdot |_{\dot{H}^{s}}$ to denote either $ | \cdot |_{\dot{H}_{dq}^{s}}$ or $| \cdot |_{\dot{H}_{F}^{s}}$, unless a distinction between them is necessary.

\section{Main results}

In this section, we state the main results of the paper, using the language introduced in the last section.

We start with a proposition about elementary properties of the thermal equilibrium measure.

\begin{proposition}\label{expdecay}
Assume $V$ is admissible and let $\theta_{0} \in \mathbf{R}^{+}.$
There exists a constant $C_{\theta_{0}}>0$ such that,
for every $\theta > \theta_{0},$ 
\begin{equation}
    \mu_{\theta}(x) \leq C_{\theta_{0}} \mbox{ for every } x \in 
    \mathbf R^d,
\end{equation}
and outside of a bounded set $\Omega$
(which depends only on $V$ and $\theta_{0}$),
\begin{equation}
    \mu_{\theta}(x) \leq \exp \left( C_{\theta_{0}} - \theta V(x) \right).
\end{equation}
\end{proposition}

The proof is found in Section ~\ref{sect:mutheta}.

We proceed to stating generalized transport inequalities, which are extensions of Theorems 1.1 and 1.2 of \cite{chafai2018concentration}.

\begin{theorem}
\label{T:transportineq}
Let $g$ be a Riesz-type kernel of order $s$, let $V$ be an admissible potential, and let $\alpha > s$.

\begin{itemize}

    \item[1.] Let $\Omega \subset \mathbf{R}^{d}$ be a compact set, then there exists a constant $C_{\alpha,\Omega}$, which depends only on $g,\alpha$ and $\Omega$ such that for any $\mu, \nu \in \mathcal{P}(\Omega)$
\begin{equation}
  \|\mu - \nu\|_{\dot{C}^{0,\alpha}_{*}}^{2} \leq C_{\alpha,\Omega} \mathcal{E}( \mu - \nu).
\end{equation}

    \item[2.] There exists a constant $C_{\alpha}$, which depends only on $V,g$, and $\alpha$ such that for every $\mu \in \mathcal{P}(\mathbf{R}^{d})$,
\begin{equation}
  \| \mu-\mu_{\infty} \|_{{C}^{0,\alpha}_{*}}^{2}  \leq C_{\alpha} \left( \mathcal{E}_{V}(\mu) -  \mathcal{E}_{V}(\mu_{\infty})  \right).
\end{equation}

    \item[3.] Assume that 
\begin{equation}
    \liminf_{x \to \infty} \frac{V(x)}{|x|^{2 \alpha}} > 0.
\end{equation}
Then there exists a constant $C_{\alpha}$, which depends only on $V,g$, and $\alpha$  such that for every $\mu \in \mathcal{P}(\mathbf{R}^{d}),$ we have
\begin{equation}
   \| \mu-\mu_{\infty} \|_{\dot{C}^{0,\alpha}_{*}}^{2}  \leq C_{\alpha} \left( \mathcal{E}_{V}(\mu) - \mathcal{E}_{V}(\mu_{\infty}) \right).
\end{equation}
\end{itemize}
\end{theorem}
The proof is found in Section ~\ref{sect:transport}.

We then apply these results to obtain a quantitative understanding of the convergence of the thermal equilibrium measure to the equilibrium measure. 

\begin{corollary}
\label{Cor:cor}
Let $V$ be an admissible potential, let $g$ be a Riesz-type kernel of order $s$, and let $\alpha > s$. Let $\theta_{n}$ be a strictly positive sequence tending to infinity, and let
\begin{equation}
    \underline{\theta} := \min_{n} \theta_{n}.
\end{equation}

Then there exists a constant $C_{\alpha}$, which depends only on $V,g$, and $\alpha$ such that
\begin{equation}
    \| \mu_{\infty} - \mu_{\theta_{n}} \|_{{C}^{0,\alpha}_{*}}^{2} \leq \frac{C_{\alpha}}{\theta_{n}} \left| {\rm ent}[\mu_{\infty}] -  {\rm ent}[\mu_{ \underline{\theta}}] \right|.
\end{equation}

Additionally, if
\begin{equation}
    \liminf_{x \to \infty} \frac{V(x)}{|x|^{2\alpha}} > 0
\end{equation}
then there exits a constant $C_{\alpha}$, which depends only on $V,g$, and $\alpha$ such that
\begin{equation}
    \| \mu_{\infty} - \mu_{\theta_{n}} \|_{\dot{C}^{0,\alpha}_{*}}^{2} \leq \frac{C_{\alpha}}{\theta_{n}} \left| {\rm ent}[\mu_{\infty}] -  {\rm ent}[\mu_{ \underline{\theta}}] \right|.
\end{equation}
\end{corollary}
The proof is found in Section~\ref{sect:transport}.

We then use the previous results to prove concentration of measure around the equilibrium measure and thermal equilibrium measure. The next theorem is a generalization of \cite{chafai2018concentration} and \cite{padilla2020concentration} to general interactions. 

\begin{theorem}
\label{T:Concentration}
Assume $V$ is an admissible potential, let $g$ be a Riesz-type kernel of order $s$, and let $\alpha > s$. Assume that $N \beta \to \infty$ and let $r > 0$. Then
\begin{itemize}

    \item[1.] There exists a constant $C$, which depends only on $V,g$, and $\alpha$ such that 
\begin{equation}
\begin{split}
    &\mathbf{P}_{N, \beta} \left( \| {\rm emp}_{N} - \mu_{\infty} \|_{{C}^{0,\alpha}_{*}} > r \right) \leq \\
    &\exp \Big( - N^{2} \beta \left( C[r-N^{-\frac{\alpha}{d}}]_{+}^{2} - C \|\mu_{\infty} \|_{L^{\infty}} N^{-\frac{2s}{d}} \right) +\\
    & \quad N \left( -\log \left| \Sigma \right| + \rm{ent}[\mu_{\infty}] + \beta \mathcal{E}(\mu_{\infty}) \right) + o(N) \Big),
\end{split}
\end{equation}
where $(\cdot)_{+}$ denotes the positive-part function. 

    \item[2.] Set $\theta = N \beta$, then there exists a constant $C$, which depends only on $V,g$, and $\alpha$ such that 
\begin{equation}
    \mathbf{P}_{N, \beta} \left( \| {\rm emp}_{N} - \mu_{\theta} \|_{{C}^{0,\alpha}_{*}} > r \right) \leq \exp \left( -N^{2} \beta \left( C\left[ r - C N^{-\frac{\alpha}{d}} \right]_{+}^{2} - C N^{-\frac{2s}{d}} \right) \right).
\end{equation}
\end{itemize}
\end{theorem}
The proof is found in Section~\ref{sect:concentration}.

The previous theorem implies, with high probability, an upper bound for the distances between the empirical measure and equilibrium and thermal equilibrium measures. We complement this with a lower bound for such distance.

\begin{proposition}
\label{Prop:opt}
Let $\mu_{N}$ be a sequence of probability measures on $\mathbf{R}^{d}$ such that
\begin{equation}
    \sup_{N} \|\mu_{N}\|_{L^{\infty}} < \infty.
\end{equation}

Let $\nu_{N}$ be a sequence of probability measures on $\mathbf{R}^{d}$ such that
\begin{equation}
    \nu_{N} = \frac{1}{N} \sum_{i=1}^{N} \delta_{x_{i}}.
\end{equation}

Then
\begin{equation}
    \|\mu_{N} - \nu_{N} \|_{{C}^{0,\alpha}_{*}} \geq C N^{-\frac{\alpha}{d}}, 
\end{equation}
where $C$ depends on $d$ and $\sup_{N} \|\mu_{N}\|_{L^{\infty}}$.
\end{proposition}
The proof is found in Section~\ref{sect:optimality}.

Our final result is a series of Moser-Trudinger like inequalities, which may also be interpreted as bounds on the Laplace transform of fluctuations. This result is new also in the Coulomb case. 

\begin{theorem}
\label{theo:MTineq}
Let $V$ be an admissible potential, and assume $N \beta \to \infty$. Let $f:\mathbf{R}^{d} \to \mathbf{R}$ be continuous, and define the random variable ${\rm Fluct}[f]$ by 
\begin{equation}
    {\rm Fluct}[f] = \int_{\mathbf{R}^{d}} f d\left( {\rm emp}_{N} - \mu_{\theta} \right).
\end{equation}

\begin{itemize}
    \item[1.]If $g$ is the Coulomb kernel, then there exists a constant $C$, which depends only on $V$ such that
\begin{equation}
    \log \left( \mathbf{E}_{\mathbf{P}_{N, \beta}} \exp \left( N^{2} \beta \left| t{\rm Fluct}[f] \right| \right) \right) \leq N^{2} \beta \left( C t^{2}\|f\|_{W^{1, \infty}}^{2} + N^{-\frac{2}{d}} C \right).
\end{equation}

    \item[2. ]If $g$ is a Riesz-type kernel of order $s$ then for any $\alpha > s$ there exists a constant $C$, which depends only on $V$, $g$, and $\alpha$ such that
\begin{equation}
    \log \left( \mathbf{E}_{\mathbf{P}_{N, \beta}} \exp \left( N^{2} \beta \left| t{\rm Fluct}[f] \right| \right) \right) \leq  N^{2} \beta \left( C t^{2}\|f\|_{{C}^{0,\alpha}}^{2} + N^{-\frac{2 \alpha}{d}} C \right).
\end{equation}

    \item[3.]If $g$ is the Coulomb kernel, $i \in \{0,1\}$ and $\alpha \in (0,1)$ then there exists a constant $C$, which depends only on $V$ such that
\begin{equation}
\begin{split}
    &\log \left( \mathbf{E}_{\mathbf{P}_{N, \beta}} \exp \left( N^{2} \beta \left| t{\rm Fluct}[f] \right| \right) \right) \leq\\
    &\beta N^{2} \left( \frac{1}{4} t^{2}|f|_{\dot{H}^{1}}^{2} + N^{- \frac{i+\alpha}{d}} t  |f|_{\dot{C}^{i, \alpha}} + C N^{-\frac{2}{d}} \right)+ \frac{d}{2}\left( \log (N^{2} \beta) + \log (1+ t|f|_{\dot{H}^{1}}) \right). 
\end{split}    
\end{equation}

    \item[4.]If $g$ is a Riesz-type kernel of order $s$, $i \in \{0,1\}$ and $\alpha \in (0,1)$, then there exists a constant $C$ depending only on $V$ and $g$ such that
\begin{equation}
    \log \left( \mathbf{E}_{\mathbf{P}_{N, \beta}} \exp \left( N^{2} \beta \left| t{\rm Fluct}[f] \right| \right) \right) \leq \beta N^{2} \left( C t^{2} |f|_{\dot{H}^{s}}^{2} + C N^{- \frac{i+\alpha}{d}} t |f|_{\dot{C}^{i, \alpha}} + C N^{-\frac{2s}{d}} \right). 
\end{equation}
\end{itemize}
\end{theorem}
The proof is found in Section~\ref{sect:MTineq}.

\section{Further work}

\begin{itemize}

    \item[1.] We expect that an analogue of the main results in this paper holds for compact manifolds. In this case, the Riesz kernel would be defined as the fundamental solution of the fractional Laplacian, and may be written in terms of the eigenfunctions of the Laplace-Beltrami operator. 

    \item[2.] Theorems \ref{T:Concentration} and \ref{T:transportineq} hold for $\alpha > s$. The reason is in Remark \ref{rem:HsCalphaembed}.  Do Theorems \ref{T:Concentration} and \ref{T:transportineq} hold also in the endpoint case $\alpha = s$? We emphasise that this s true in the case $s=1$.
    
    \item[3.] Do transport inequalities hold for the thermal equilibrium measure as well? That is, is it true (in the Coulomb case) that
    \begin{equation}
        \begin{split}
            \| \mu - \mu_{\beta} \|_{\rm BL} &\leq C \left( \mathcal{E}^{V}_{\beta}(\mu) - \mathcal{E}^{V}_{\beta}(\mu_{\beta}) \right), \\
           W_{1} (\mu, \mu_{\beta}) &\leq C \left( \mathcal{E}^{V}_{\beta}(\mu) - \mathcal{E}^{V}_{\beta}(\mu_{\beta}) \right),
        \end{split}
    \end{equation}
    for some constant $C$, where $\|\cdot \|_{\rm BL}$ and $ W_{1}$ denote the bounded-Lipschitz norm and Wasserstein $1-$distance, respectively? These inequalities would prove a more intimate link between optimal transport and Coulomb gases than the ones available in the literature. 
    
    \item[4.] Theorem $\ref{T:Concentration}$ implies, with high probability, an upper bound for the bounded-Lipschitz distance between the empirical measure and thermal equilibrium measure. However, unlike the Coulomb case, this inequality is not geometrically optimal, i.e. it is not of order $N^{-\frac{1}{d}}$. Is it true that
    \begin{equation}
    \label{eq:optBLdist}
        \mathbf{P}_{N, \beta} (\limsup N^{\frac{1}{d}} \| {\rm emp} - \mu_{\beta} \|_{\rm BL} = \infty) =0?
    \end{equation}
    In the case of Coulomb interactions, it is shown in \cite{padilla2020concentration} that equation \eqref{eq:optBLdist} holds, and that it cannot hold for any exponent higher than $\frac{1}{d}$. 
    
    \item[5.] We expect that generalized transport inequalities and concentration inequalities are also valid in the high temperature regime ($\beta = \frac{1}{N}$); and very high temperature regime ($\beta \ll \frac{1}{N}$). In the case of the very high temperature regime, it would be necessary to consider either a compact manifold, or a Hamiltonian in which the confining potential term has weight big enough to be comparable to the effect of the entropy.
    
    \item[6.] We expect that generalized transport inequalities (Theorem \ref{T:transportineq}) and concentration inequalities (Theorem \ref{T:Concentration}) also hold for subCoulomb interactions, i.e. interactions of the form $g(x) = |x|^{d-2s}$ for $d\geq 3$ and $s \in \left( 1, \frac{d}{2} \right)$. In this case, the inequalities would be for the $C^{k,\alpha}_{*}$ or $\dot{C}^{k,\alpha}_{*}$ norms, for $k \in \left( 0, \frac{d}{2} \right) \cap \mathbf{N}$.
\end{itemize}

\section{Literature comparison}

The thermal equilibrium measure was introduced in this context in \cite{armstrong2019thermal} and \cite{armstrong2021local}, and further explored in \cite{serfaty2020gaussian}. \cite{armstrong2019thermal} analyzes fundamental qualitative and quantitative properties, while \cite{armstrong2021local} introduces the splitting formula, and exploits it to obtain local laws for Coulomb gases at arbitrary temperature. \cite{serfaty2020gaussian} Continues the ideas of \cite{armstrong2021local} by deriving a precise expansion of the partition function-equation \eqref{eq:partfunc}, and applying these estimates to obtain a CLT for fluctuations around the thermal equilibrium measure. This paper begins the investigation of the thermal equilibrium measure for Riesz and more general interactions. 

The type of interactions that we treat in this paper are, roughly speaking, interactions that qualitatively behave like Riesz potentials in both real space and Fourier space. Very similar kernels were introduced in \cite{nguyen2021mean} and \cite{rosenzweig2021global}. In the last two references, the goal was to understand the passage to the mean-field limit of a system of ODE's modelling the dynamic behaviour of a many-particle system. More specifically, the authors show that if $X_{N} \in \mathbf{R}^{d \times N}$ satisfies the ODE
\begin{equation}
\begin{split}
    \dot{x}_{i} &= \frac{1}{N} \sum_{i \neq j } \nabla g (x_{i} - x_{j})\\
    x_{i}(0) &= x_{i}^{0},
\end{split}    
\end{equation}
then in the limit as $N$ tends to infinity, ${\rm emp}_{N}(X_{N}(0))$ converges to a measure $\mu_{0}$, and under additional conditions, then ${\rm emp}_{N}(X_{N}(t))$ converges to a solution of the PDE
\begin{equation}
    \begin{split}
        \partial_{t} \mu &= - {\rm div}((\nabla g \ast \mu) \mu ) \\
        \mu(0) &= \mu_{0}.
    \end{split}
\end{equation}
The exact result is more general, since it covers possible stochastic noise, among other things, see \cite{nguyen2021mean,rosenzweig2021global} for further details.

The Moser-Trudinger-type inequalities of Theorem \ref{theo:MTineq} (which may also be interpreted as bounds on the Laplace transform of the fluctuations) may be compared to Theorem 1 of \cite{serfaty2020gaussian}, and Theorems 1.2 and 1.5 of \cite{berman2019sharp}. Theorems 1.2 and 1.5 of \cite{berman2019sharp} apply only to two-dimensional Coulomb gases, but they are sharp in that case. On the other hand the inequalities of Theorem \ref{theo:MTineq} hold in arbitrary dimension and temperature regime, but we do not believe them to be sharp. Theorem 1 of \cite{serfaty2020gaussian} also holds in arbitrary dimension and temperature regime. It contains a bound in terms of the maximum absolute value of $f$ and its higher order derivatives (or $\xi$ in the notation of \cite{serfaty2020gaussian}). This bound requires higher regularity (namely, $C^{3}$ regularity), and also that $C t \max \{|f|_{C^{1}}, |f|_{C^{2}}\} <1$ for a given constant $C$, but is more accurate than Theorem \ref{theo:MTineq} for small values of $t$. Unlike the bounds in both \cite{serfaty2020gaussian} and \cite{berman2019sharp}, Theorem \ref{theo:MTineq} holds for Riesz kernels and more general interactions.

\cite{sandier20152d, rougerie2016higher} obtain concentration inequalities for Coulomb gases as a consequence of an analysis of the renormalized energy. \cite{leble2018fluctuations} obtains bounds on the Laplace transform of fluctuations for Coulomb gases in 2d, which imply concentration inequalities. \cite{bekerman2018clt} obtains bounds on the Laplace transform of fluctuations and also concentration bounds for log gases in 1d. \cite{leble2017large} derives and LDP for Coulomb and Riesz gases; in particular, the upper bound of the LDP may be interpreted as a concentration inequality in which the error depends on the set in question. We expect that concentration inequalities for Riesz gases may also be derived using the ideas of \cite{petrache2017next}.

This paper is inspired by \cite{chafai2018concentration,maida2014free}.
 \cite{maida2014free} is an early reference in the connection between optimal transport and Coulomb gases. This work was later expanded in \cite{chafai2018concentration}. The main results in \cite{chafai2018concentration} are analogues of Theorems \ref{T:transportineq} and \ref{T:Concentration} item 1 for the Coulomb kernel, with the $C^{0,\alpha}_{*}$ norm replaced by the bounded-Lipschitz norm, and the $\dot{C}^{0,\alpha}_{*}$ replaced by the Wasserstein 1-norm. This concentration inequality was later extended to compact manifolds in \cite{garcia2019concentration}, and to the thermal equilibrium measure in \cite{padilla2020concentration}. In this paper, we present an alternative approach to the one in \cite{chafai2018concentration}. This approach, based on splitting formulas, a dual representation of the electric energy, and a localization inequality for the $H^{-s}$ norm (see Proposition \ref{localizingH-snorm}) allows us to recover the results of \cite{chafai2018concentration} and extend them to more general interactions. 

To the best of our knowledge, this is the first concentration inequality for Riesz gases (at least the first one in which the error is independent of the set in question). 

\section{Preliminaries}
\label{preliminaries}

This section will lay the groundwork needed to prove the results introduced earlier. 

\subsection{Splitting formulas and partition functions}

We start by recalling some well known facts about the equilibrium and thermal equilibrium measures.

First, we recall existence and uniqueness. 
\begin{lemma}
Assume $V$ is admissible and $g$ is either the Coulomb kernel or a Riesz-type kernel. Then the functional $\mathcal{E}_{V}$ has a unique minimizer in the set of probability measures, which has compact support $\Sigma$ and satisfies the First Order Condition
\begin{equation}\label{ELeq}
    \begin{split}
        h^{\mu_{\infty}} + V - c_{\infty} &\geq 0 \text{ in } \mathbf{R}^{d}\\
        h^{\mu_{\infty}} + V - c_{\infty} &= 0 \text{ in } \mathrm{supp}(\mu_{\infty}), 
    \end{split}
\end{equation}
for some constant $c_{\infty} \in \mathbf{R}.$
\end{lemma}

\begin{proof}
See \cite{frostman1935potentiel, serfaty2018systems} for the Coulomb case, which extends to the Riesz-type case.
\end{proof}

\begin{lemma}
The functional $\mathcal{E}_{V}^{\theta}$ has a unique minimizer in the set of probability measures, which is everywhere positive, bounded from above (with a bound that may depend on $\theta$), satisfies the First Order Condition
\begin{equation}
\label{eq:ELeqtherm}
    h^{\mu_{\theta}} + V + \frac{1}{\theta} \log (\mu_{\theta}) = c_{\theta},
\end{equation}
for some constant $c_\theta$, and also that
\begin{equation}
\label{eq:decayofpot}
    \lim_{|x| \to \infty} h^{\mu_{\theta}} (x)=0.
\end{equation}
\end{lemma}

\begin{proof}
 See \cite{armstrong2019thermal} for the Coulomb case, which extends to the Riesz-type case.
\end{proof}

Next, we derive expansions of the Hamiltonian around the equilibrium and thermal equilibrium measures. These expansions are splitting formulas. 
   
\begin{proposition}[Splitting formula]
\label{Prop:splitting}
For every $X_N = (x_1,\dots,x_N) \in \mathbf R^{d \times N}$ 
such that $x_i \neq x_j$ whenever $i \neq j$, we have that the Hamiltonian $\mathcal{H}_{N}$ can be split into 
\begin{equation}
\label{eq:splittingform}
    \mathcal{H}_{N}(X_{N}) = N^{2} \left( \mathcal{E}_{V}(\mu_{\infty}) + {\rm F}_{N}(X_{N}, \mu_{\infty}) + \int_{\mathbf{R}^{d}} \zeta_{\infty}\, \mathrm d\, {\rm emp}_{N} \right),
\end{equation}
where 
\begin{equation}
    \zeta_{\infty} = V +c_{\infty}- 2 h^{\mu_{\infty}} .
\end{equation}
\end{proposition}

\begin{proof}
See \cite{serfaty2018systems} for the Coulomb case, which extends to the Riesz-type case.
\end{proof}

\begin{proposition}[Thermal splitting formula]
We introduce the notation
\begin{equation}
    \zeta_{\theta} = - \frac{1}{\theta} \log(\mu_{\theta}).
\end{equation}
For every $X_N = (x_1,\dots,x_N) \in \mathbf R^{d \times N}$ 
such that $x_i \neq x_j$ whenever $i \neq j$, we have that the Hamiltonian $\mathcal{H}_{N}$ can be split into  
\begin{equation}\label{eq:thermspltfrm}
    \mathcal{H}_{N} (X_{N}) = N^{2} \left( \mathcal{E}_{V}^{\theta} (\mu_{\theta}) +  {\rm F}_{N}(X_{N}, \mu_{\theta})+ \int_{\mathbf{R}^{d}} \zeta_{\theta}\, \mathrm d\, {\rm emp}_{N} \right),
\end{equation}
with
\begin{equation}
    \theta = N \beta. 
\end{equation}
\end{proposition}

\begin{proof}
See \cite{armstrong2021local} for the Coulomb case, which extends to the Riesz-type case.
\end{proof}

Motivated by splitting formulas, we can identify the leading order term
of the partition function, and also prove elementary results about the next order components.  

\begin{lemma}
\label{lem:partfunc}
Assume that $V$ is admissible, finite on a set of positive Lebesgue measure, and that $\rm{ent}[\mu_{\infty}] < \infty$. Then for any $N \geq 2$,
\begin{equation}
    Z_{\beta}^{N} \geq \exp \left( -N^{2}\beta \mathcal{E}_{V}(\mu_{\infty}) + N \left( \beta \mathcal{E}(\mu_{\infty}) + \rm{ent}[\mu_{\infty}] \right) \right). 
\end{equation}
\end{lemma}

\begin{proof}
See \cite{chafai2018concentration}, Lemma 4.1 for the Coulomb case, which extends to the Riesz-type case.
\end{proof}

\begin{definition}
Set $\theta = N \beta$. We define the next order partition function $K_{N, \beta}$ as 
\begin{equation}
    K_{N, \beta} = \frac{Z_{N, \beta}}{ \exp \left( N^{2} \mathcal{E}_{V}^{\theta}(\mu_{\theta}) \right)}.
\end{equation}
\end{definition}

\begin{lemma}
\label{lem:part}
Let $V$ be an admissible potential, and let $g$ be either a Riesz-type kernel or the Coulomb kernel. Then the next order partition function is greater that $1$, i.e., 
\begin{equation}
    \log( K_{N, \beta}) > 0.
\end{equation}
\end{lemma}

\begin{proof}
See \cite{padilla2020concentration, armstrong2021local} for the Coulomb case, which extends to the Riesz-type case.
\end{proof}

\subsection{More on $H^{s}$ norms}

In this subsection, we introduce additional properties of the $H^{s}$ norm.  
We start by turning the $\dot{H}^{s}$ semi-norm into a full norm. 

\begin{definition}
We define the $H^{s}$ norm of a function $f: \mathbf{R}^{d} \to \mathbf{R}$ by \begin{equation}
    \|f\|_{H^{s}}^{2} = \|f\|_{L^{2}}^{2} + |f|^{2}_{\dot{H}^{s}}
\end{equation}
\end{definition}

We now extend the definition of $H^{s}$ norms to negative indices. 

\begin{definition}
Given a Radon measure $\mu$ on $\mathbf{R}^{d},$ and $s > 0$ we define the $\dot{H}^{-s}$ norm of $\mu$ by
\begin{equation}
    \|\mu\|_{\dot{H}^{-s}}^{2} =  \int_{\mathbf{R}^{d}} |\widehat{\mu}(\xi)|^{2}
    |\xi|^{-2s} d \xi.
\end{equation}
\end{definition}

Now we introduce an analogue of the $H^{s}$ norms for functions that are not defined on the whole space.
\begin{definition}
Given a domain $\Omega,$ and a measurable function $f: \Omega \to \mathbf{R}$ we define the $\dot{H}^{s}$ semi-norm restricted to $\Omega$ by
\begin{equation}
      |f|_{\dot{H}^{s}(\Omega)}^{2} = \iint_{\Omega \times \Omega} \frac{|f(x) - f(y)|^{2}}{|x-y|^{d + 2s}} \, \mathrm d x
      \mathrm d y.
\end{equation}

and the ${H}^{s}(\Omega)$ norm by
\begin{equation}
    \| f \|_{{H}^{s}(\Omega)}^{2} = \|f\|_{L^{2}}^{2} +  |f|_{\dot{H}^{s}(\Omega)}^{2}.
\end{equation}
\end{definition}

An alternative way to define an $H^{s}$ space on a bounded set is to take restrictions of $H^{s}$ functions on the whole space. These two approaches are equivalent, and this is the focus of the next lemma. 
\begin{lemma}
\label{lem:equivdefs}
Given a domain $\Omega$ with a smooth boundary and $f : \Omega \to \mathbf{R},$ the following are equivalent:
\begin{itemize}
    \item[1.]
    \begin{equation}
         |f|_{\dot{H}^{s}(\Omega)} < \infty
    \end{equation} 
    
    \item[2.] There exists $\overline{f} : \mathbf{R}^{d} \to \mathbf{R}$ such that $\overline{f} |_{\Omega} = f$ and 
    \begin{equation}
         | \overline{f}|_{\dot{H}^{s}} < \infty.
    \end{equation}
\end{itemize}
\end{lemma}

\begin{proof}
See \cite{triebel2006theory}.
\end{proof}

Next, we introduce the extension operator for $H^{s}$ functions. 

\begin{lemma}[Extensions]
\label{lem:extension}
Let $\Omega \subset \mathbf{R}^{d}$ be a compact domain with a smooth boundary. There exists a linear operator from $H^{s}(\Omega)$ to $\dot{H}^{s}(\mathbf{R}^{d})$,  which takes a function $f \in H^{s}(\Omega)$ to a function $ \overline f \in \dot{H}^{s}(\mathbf{R}^{d})$ that satisfies
\begin{itemize}
    \item[1.] 
    $
        \overline{f}|_{\Omega} = f \mbox{ and}$

    \item[2.]
    $
        \| \overline{f} \|_{{H}^{s}} 
        \leq C \| {f} \|_{{H}^{s}(\Omega)},
    $
\end{itemize}
where $C$ depends on $\Omega$ but does not depend on $f$.

\end{lemma}

\begin{proof}
See \cite{triebel2006theory}.
\end{proof}

We also introduce the $H^{-s}(\Omega)$ norm, analogous to the $H^{-1}(\Omega)$ norm introduced in \cite{padilla2020concentration}.

\begin{definition}
Given a bounded open set $\Omega,$ a Radon measure $\mu$ on $\Omega$, we define the $H^{-s}(\Omega)$ norm as 
\begin{equation}
    \|\mu \|_{H^{-s}(\Omega)} = \sup \frac{ \int_{\Omega}  \phi d \mu }{\| \phi \|_{H^{s}(\Omega)}}.
\end{equation}
\end{definition}

A key point in this paper is that the $H^{s}$ and $C^{0,\alpha}$ norms can be easily compared in bounded sets.
\begin{remark}
\label{rem:HsCalphaembed}
Note that, for any $s < \alpha$ and any bounded set $\Omega \subset \mathbf{R}^{d}$, there exists $C>0$ such that for
any $f : \Omega \to \mathbf{R}^{d}$
\begin{equation}
    | f |_{\dot{H}^{s}} \leq C | f |_{\dot{C}^{0,\alpha}},
\end{equation}
which implies that for any Radon measure $\mu$ on $\Omega$,
\begin{equation}
    \| \mu \|_{{C}^{0,\alpha}_{*}} \leq C \| \mu \|_{H^{-s}(\Omega)}.
\end{equation}
\end{remark}

We end this section with a lemma that relates the $H^{-s}$ norm of a measure to its electric energy. 

\begin{lemma}
Let $g$ be the Riesz kernel of order $s$ (equation \eqref{Rieszcase}) or the Coulomb kernel (in which case we take $s=1$), and let $\mu$ be a Radon measure on $\mathbf{R}^{d}$. Then 
\begin{equation}
    \mathcal{E}(\mu) =  \|\mu\|_{\dot{H}^{-s}}^{2} = \left( \sup_{\phi \in C^{\infty}_{0}} \frac{\int_{\mathbf{R}^{d}} \phi \, d \mu}{| \phi |_{\dot{H}_{F}^{s}}} \right)^{2}.
\end{equation}
\end{lemma}

\begin{proof}
It is clear that 
\begin{equation}
     \mathcal{E}(\mu) =  \|\mu\|_{\dot{H}^{-s}}^{2} . 
\end{equation}

By taking as a test functions a sequence of smooth functions with compact support that converge to $ h^{\mu}$, we also have that 
\begin{equation}
     \sup_{\phi \in C^{\infty}_{0}} \frac{\int_{\mathbf{R}^{d}} \phi \, d \mu}{| \phi |_{\dot{H}_{F}^{s}}} \geq \frac{\int_{\mathbf{R}^{d}} h^{\mu} \, d \mu}{| h^{\mu} |_{\dot{H}_{F}^{s}}} = \|\mu\|_{\dot{H}^{-s}}.
\end{equation}

Hence, we will prove that 
\begin{equation}
    \sup_{\phi \in C^{\infty}_{0}} \frac{\int_{\mathbf{R}^{d}} \phi \, d \mu}{| \phi |_{\dot{H}_{F}^{s}}} \leq \|\mu\|_{\dot{H}^{-s}}.
\end{equation}

This is because, for any $ \phi \in C^{\infty}_{0},$ we have that
\begin{equation}
    \begin{split}
        \int_{\mathbf{R}^{d}} \phi \, d \mu &= \int_{\mathbf{R}^{d}} \widehat{\phi}  \widehat{\mu} \, d \xi \\
        &= \int_{\mathbf{R}^{d}} |\xi|^{s}\widehat{\phi}  \frac{1}{|\xi|^{s}}\widehat{\mu} \, d \xi \\
        &\leq \| |\xi|^{s}\widehat{\phi} \|_{L^{2}} \|  {|\xi|^{-s}}\widehat{\mu} \|_{L^{2}} \\
        &= | \phi |_{\dot{H}_{F}^{s}} \| \mu \|_{\dot{H}^{-s}}.
    \end{split}
\end{equation}

From this, we can conclude.
\end{proof}

\section{Proof of statements about the thermal equilibrium measure}
\label{sect:mutheta}

This section is devoted to proving Proposition \ref{expdecay}. The proof will be a consequence of a series of lemmas.

\begin{lemma}
Let $g$ be a Riesz-type kernel of order $s$ and let $\mathfrak{D}(f)$ be the inverse of the operator
\begin{equation}
    \mu \mapsto h^{\mu}. 
\end{equation}
Note that the inverse is a well defined map on from $H^{2s}$ to $L^{2}$, since $\|h^{\mu}\|_{L^{2}}$ is equivalent to $\|\mu \|_{H^{-2s}}$ by item 8 of Definition \ref{def:riesztype}. Then $\mathfrak{D}(f)$ can be written as 
\begin{equation}\label{eq:inverse}
    \mathfrak{D}(f) (x) =  \frac{1}{2} \int_{\mathbf{R}^{d}} \left( -2f(x)+f(x+y) + f(x-y) \right) \left( h(y) \right) \, \mathrm d y,
\end{equation}
where $h$ is given by  
\begin{equation}
    \widehat{h} = \frac{1}{\widehat{g}}.
\end{equation}
\end{lemma}

\begin{remark}
Note that the integral in equation \eqref{eq:inverse} is well defined for any $f \in H^{2s}$ because of item 10 in the definition of Riesz-type kernel, Definition \ref{def:riesztype}.
\end{remark}

\begin{proof}
We begin by taking the Fourier transform in $x$:
\begin{equation}
        \widehat{\mathfrak{D}(f)} (\xi) = \frac{1}{2} \int_{\mathbf{R}^{d}} \left( -2+ e^{i \xi \cdot y} + e^{-i \xi \cdot y} \right)  h(y) \widehat{f}(\xi) \, \mathrm d y.
\end{equation}

We now need to show that 
\begin{equation}
\begin{split}
    F(\xi) &:= \frac{1}{2} \int_{\mathbf{R}^{d}} \left( -2+ e^{i \xi \cdot y} + e^{-i \xi \cdot y} \right)  h(y) \, \mathrm d y\\
    &= \frac{1}{\widehat{g}(\xi)}.
\end{split}    
\end{equation}

To prove this, we apply Plancherel's Theorem: 
\begin{equation}
        F(\xi) = \frac{1}{2} \langle -2\delta_{0} + \delta_{\xi} + \delta_{-\xi}, \widehat{h} \rangle.
\end{equation}
Since $\widehat{h}$ has a well-defined, finite value at $0$, we have that 
\begin{equation}
    \langle \delta_{0}, \widehat{h} \rangle=0.
\end{equation}
On the other hand, since $g$ is even (item $2$ of definition \ref{def:riesztype}), $\widehat{h}$ is even, and therefore
\begin{equation}
    \frac{1}{2} \langle \delta_{\xi} + \delta_{-\xi}, \widehat{h} \rangle= \widehat{h} (\xi).
\end{equation}

From this we can conclude. 
\end{proof}

\begin{corollary}
As a consequence of the previous lemma, the operator $\mathfrak{D}$ satisfies a monotonicity property, i.e. if $x$ is a maximum of $f$, then $\mathfrak{D}f(x) \leq 0$. Furthermore, if the inequality is not strict, then $f$ is constant a.e. 
\end{corollary}

After having access to a monotonicity property, we can now follow the same strategy as in \cite{armstrong2019thermal} to proof Proposition \ref{expdecay}. We begin with a lemma, which is the analogue of Lemma 3.2 in \cite{armstrong2019thermal}.



\begin{lemma}
\label{lem:mutheta}
Let 
\begin{equation}
    m_{\theta} = \sup_{\mathbf{R}^{d}} \mu_{\theta}.
\end{equation}
Then
\begin{equation}
    - \frac{\log m_{\theta}}{\theta} \leq h^{\mu_{\theta}} - c_{\theta} - (h^{\mu_{\infty}} - c_{\infty}).
\end{equation}
\end{lemma}

\begin{proof}
We first show that we have
\begin{equation}
    \liminf_{|x| \to \infty} \left( h^{\mu_{\theta}} + c_{\infty} - c_{\theta} + \frac{\log m_{\theta}}{\theta} \right)  \geq 0,
\end{equation}
which in virtue of equation \eqref{eq:decayofpot} is equivalent to showing that 
\begin{equation}
    c_{\infty} - c_{\theta} + \frac{\log m_{\theta}}{\theta} \geq 0.
\end{equation}

To prove this claim, we proceed by contradiction and assume that  assume that 
\begin{equation}
    c_{\infty} - c_{\theta} + \frac{\log m_{\theta}}{\theta} < 0.
\end{equation}
We define $\psi$ as the unique function satisfying that
\begin{equation}
\begin{split}
   \mathfrak{D} \psi = 0  \quad  {\rm in } \quad  \mathbf{R}^{d} \setminus \Sigma\\
   \psi = 0 \quad  {\rm in } \quad  \Sigma\\
  \lim_{x \to \infty} \psi (x) = c_{\infty} - c_{\theta} + \frac{\log m_{\theta}}{\theta}.
\end{split}   
\end{equation}

Because $ c_{\infty} - c_{\theta} + \frac{\log m_{\theta}}{\theta} < 0$ by hypothesis, we have $\psi - (c_{\infty} - c_{\theta} + \frac{\log m_{\theta}}{\theta})$ decays at infinity like the Green's function associated to the operator $\mathfrak{D}$, i.e. for $x$ large enough, $\psi (x) - (c_{\infty} - c_{\theta} + \frac{\log m_{\theta}}{\theta})$ is bounded above and below by a constant multiple of $|x|^{2s - d}$ by item 5 of Definition \ref{def:riesztype}.

On the other hand, we define 
\begin{equation}
    \varphi := h^{\mu_{\theta}} - h^{\mu_{\infty}} + c_{\infty} - c_{\theta} + \frac{\log m_{\theta}}{\theta}.
\end{equation}
Note that $\mu_{\theta}$ satisfies 
\begin{equation}
    h^{\mu_{\theta}} + V - c_{\theta} + \frac{\log m_{\theta}}{\theta} \geq 0,
\end{equation}
and because of the First Order Condition for $\mu_{\infty}$-equation \eqref{ELeq}, we have that $\varphi$ satisfies 
\begin{equation}
\label{eq:defofphi}
\begin{split}
    \varphi \geq 0 \quad {\rm in} \quad \Sigma\\
    \mathfrak{D}(\varphi) \geq 0 \quad {\rm in} \quad \mathbf{R}^{d} \setminus \Sigma.
\end{split}
\end{equation}

We then have that $\mathfrak{D}( \varphi - \psi ) \geq 0$ in $\mathbf{R}^{d} \setminus \Sigma$, $\lim_{x \to \infty}  \varphi(x) - \psi(x) =0$, and $\varphi \geq \psi$ in $\Sigma$. Then by monotonicity $\varphi - \psi \geq 0$ in $\mathbf{R}^{d} \setminus \Sigma$. Furthermore, since 
\begin{equation}
    \int_{\mathbf{R}^{d}} \mathfrak{D}( \varphi ) = \int_{\mathbf{R}^{d}} \mu_{\theta} - \mu_{\infty}=0,
\end{equation}
we also have that $\varphi - (c_{\infty} - c_{\theta} + \frac{\log m_{\theta}}{\theta})$ satisfies that for $x$ large enough, it is bounded above and below by a constant multiple of $x^{2s-d-1}$ by item 6 of Definition \ref{def:riesztype}.

This last statement, the fact that $\varphi - (c_{\infty} - c_{\theta} + \frac{\log m_{\theta}}{\theta})$ decays at infinity like $|x|^{2s - d}$, and that fact that $\varphi \geq \psi$ imply a contradiction. Therefore 
\begin{equation}
    \liminf_{x \to \infty} \varphi(x) \geq 0.
\end{equation}
Hence, by equation \eqref{eq:defofphi} and monotonicity, we deduce that $\varphi \geq 0$ in all of $\mathbf{R}^{d}$. This implies the desired result. 

\end{proof}

We can now prove Proposition \ref{expdecay}, restated here for convenience. The proof is analogous to the proof of Lemma 3.3 in \cite{armstrong2019thermal}:
\begin{proposition}\label{expdecay2}
Assume $V$ is admissible and let $\theta_{0} \in \mathbf{R}^{+}.$ For every $\theta > \theta_{0},$ there exists a constant $C_{\theta_{0}}$ such that
\begin{equation}
\label{eq:unifbound}
    \mu_{\theta}(x) \leq C_{\theta_{0}},
\end{equation}
and outside of a compact set $\Omega$ (which depends only on $V$ and $\theta_{0}$),
\begin{equation}
\label{eq:expdecay}
    \mu_{\theta}(x) \leq \exp \left( C_{\theta_{0}} - \theta V(x) \right).
\end{equation}
\end{proposition}

\begin{proof}
We start with Lemma \ref{lem:mutheta}, which along with the hypothesis that $\mu_{\infty}$ is bounded implies that 
\begin{equation}
    h^{\mu_{\theta}} - c_{\theta} \geq -C - \frac{\log m_{\theta}}{\theta},
\end{equation}
for some $C$ which depends only on $V$. 

Inserting into the First Order Condition for $\mu_{\theta}$-equation \eqref{eq:ELeqtherm}, we get that 
\begin{equation}
\label{eq:mutheta}
    \begin{split}
        \log \mu_{\theta} &= \theta (c_{\theta} - h^{\mu_{\theta}}) - \theta V \\
        &\leq \theta C + \log m_{\theta} - \theta V.
    \end{split}
\end{equation}

Since $V(x)$ tends to $\infty$ as $x$ tends to $\infty$, we have that for every $\theta_{0}$ there exists $R_{0}$ such that for every $\theta \geq \theta_{0}$ and $x \notin B(0,R_{0})$ we have 
\begin{equation}
    \log \mu_{\theta}(x) \leq \log m_{\theta} - 1.
\end{equation}
Therefore the supremum of $\mu_{\theta}$ is achieved at $B(0,R_{0})$ for every $\theta \geq \theta_{0}$. 

Let $x_{\theta} \in B(0,R_{0})$ be the maximizer of $\mu_{\theta}$. By monotonicity we have that $\mathfrak{D}(\log (\mu_{\theta})) \geq 0$, and therefore by the First Order Condition for $\mu_{\theta}$ (equation \eqref{eq:ELeqtherm}),
\begin{equation}
     \mu_{\theta} (x_{\theta}) + \mathfrak{D}V (x_{\theta}) \leq 0. 
\end{equation}
Therefore 
\begin{equation}
    m_{\theta} \leq C_{\theta_{0}} :=  \sup_{x \in B(0,R_{0})} |\mathfrak{D}V (x)|,
\end{equation}
which implies that $\mu_{\theta}$ is bounded independently of $\theta$, and therefore equation \eqref{eq:unifbound}. Note that the hypothesis that $V \in C^{2}$, along with item $10$ of a Riesz-type kernel (Definition \ref{def:riesztype}) imply that $\sup_{x \in B(0,R)} |\mathfrak{D}V (x)| < \infty$. This uniform bound, together with equation \eqref{eq:mutheta} implies that
\begin{equation}
        \log \mu_{\theta} \leq \theta C + \log C_{\theta_{0}} - \theta V,
\end{equation}
which implies the uniform exponential decay (equation \eqref{eq:expdecay}). 

\end{proof}

\section{Proof of transport inequalities}
\label{sect:transport}

This section is devoted to proving Theorem \ref{T:transportineq}. The proof will rely on the following localization inequality, which is an extension of an analogous proposition in \cite{padilla2020concentration} to the $H^{-s}$ norm. 

\begin{proposition}
\label{localizingH-snorm}
Let $\nu \in H^{-s}(\mathbf{R}^{d})$ and assume that there exists a compact set $\Omega$ such that $\nu$ is nonpositive or nonnegative outside of $\Omega$. Then there exists a compact set $\Omega_{1}$ which contains $\Omega,$ and a constant $C$ such that 
\begin{equation}
    \parallel \nu \parallel_{H^{-s}} \geq C \parallel \nu|_{\Omega_{1}} \parallel_{H^{-s}(\Omega_{1})}.
\end{equation}
Furthermore, $C$ and $\Omega_{1}$ depend only on $\Omega.$
\end{proposition} 
The proof is found in Secction~\ref{sect:appendixA}.

We now prove the first item of Theorem \ref{T:transportineq}, restated here for convenience.

\begin{theorem}
Let $g$ be a Riesz-type kernel of order $s$, let $V$ be an admissible potential, and let $\alpha > s$.
\begin{itemize}

    \item[1.] Let $\Omega \subset \mathbf{R}^{d}$ be a compact set, then there exists a constant $C_{\alpha,\Omega}$, which depends only on $g,\alpha$ and $\Omega$ such that for any $\mu, \nu \in \mathcal{P}(\Omega)$
\begin{equation}
  \|\mu - \nu\|_{\dot{C}^{0,\alpha}_{*}}^{2} \leq C_{\alpha,\Omega} \mathcal{E}( \mu - \nu).
\end{equation}
\end{itemize}
\end{theorem}

\begin{proof}

First, note that the $\dot{C}^{0,\alpha}_{*}$ norm can be rewritten as 
\begin{equation}
     \|\mu - \nu\|_{\dot{C}^{0,\alpha}_{*}}  = \sup_{|\phi|_{\dot{C}^{0,\alpha}} \leq 1 , \  \phi(0) =0} \int_{\mathbf{R}^{d}} \phi \, d (\mu - \nu).
\end{equation}

Since $\Omega$ is compact, there exists $R > 0$ such that 
\begin{equation}
    \Omega \subset B(0,R).
\end{equation}
Hence,
\begin{equation}
     \|\mu - \nu\|_{\dot{C}^{0,\alpha}_{*}}  \leq \sup_{|\phi|_{\dot{C}^{0,\alpha}}  \leq R , \ \|\phi\|_{L^{\infty}} \leq R } \int_{\mathbf{R}^{d}} \phi \, d (\mu - \nu).
\end{equation}

Note that if 
\begin{equation}
    |\phi|_{\dot{C}^{0,\alpha}}  \leq R , \ \|\phi\|_{L^{\infty}} \leq R
\end{equation}
then
\begin{equation}
     \|\phi\|_{H^{s}} \leq C,
\end{equation}
where $C$ depends on $\Omega.$ Hence
\begin{equation}
     \|\mu - \nu\|_{\dot{C}^{0,\alpha}_{*}}   \leq C \| \mu - \nu\|_{H^{-s}({\Omega})},
\end{equation}
where $C$ depends on $s,\alpha$, and $\Omega$.

Lastly, note that Lemma \ref{lem:extension} implies that
\begin{equation}
    \| \mu - \nu\|_{H^{-s}({\Omega})} \leq C \| \mu - \nu\|_{H^{-s}},
\end{equation}
where $C$ depends on $s$ and $\Omega$.

Putting everything together, and noting that $\mathcal{E}(\cdot)$ is equivalent to $\|\cdot\|^{2}_{\dot{H}^{-s}}$ by property $8$ of Riesz-type gases (Definition \ref{def:riesztype}) we have that 
\begin{equation}
     \|\mu - \nu\|_{\dot{C}^{0,\alpha}_{*}}^{2} \leq C_{\alpha,\Omega} \mathcal{E}( \mu - \nu),
\end{equation}
where $C$ depends on $s,\alpha$, and $\Omega$.
\end{proof}

We proceed to proving the second item of Theorem \ref{T:transportineq}, restated here for convenience.

\begin{theorem}
Let $g$ be a Riesz-type kernel of order $s$, let $V$ be an admissible potential, and let $\alpha > s$.

\begin{itemize}
    \item[2.] There exists a constant $C_{\alpha}$, which depends only on $V,g$, and $\alpha$ such that for every $\mu \in \mathcal{P}(\mathbf{R}^{d})$,
\begin{equation}
  \| \mu-\mu_{\infty} \|_{{C}^{0,\alpha}_{*}}^{2}  \leq C_{\alpha} \left( \mathcal{E}_{V}(\mu) -  \mathcal{E}_{V}(\mu_{\infty})  \right).
\end{equation}
\end{itemize}
\end{theorem}

\begin{proof}

Since $\mu - \mu_{\infty}$ is nonnegative outside of $\Sigma,$ by Proposition \ref{localizingH-snorm}, there exists $\Omega$ containing $\Sigma$ such that 
\begin{equation}\label{localization2}
    \parallel \mu - \mu_{\infty} \parallel_{H^{-s}} \geq C \parallel ( \mu - \mu_{\infty})\mathbf{1}_{\Omega} \parallel_{H^{-s}(\Omega)},
\end{equation}
where $C$ and $\Omega$ depends only on $V$ and $s$. We may assume WLOG that, in addition, 
\begin{equation}\label{lowerboundonxi}
    V-2 h^{\mu_{\infty}} \geq 1
\end{equation}
outside of $\Omega$ by item $2$ of Definition \ref{def:admissibleV}.

We then have that
\begin{equation}
    \begin{split}
       \| \mu-\mu_{\infty} \|_{{C}^{0,\alpha}_{*}}^{2}  &\leq 2 \left( \| \mu \mathbf{1}_{\Omega} - \mu_{\infty} \|_{{C}^{0,\alpha}_{*}}^{2} + \| \mu \mathbf{1}_{ \mathbf{R}^{d} \setminus \Omega} \|_{{C}^{0,\alpha}_{*}}^{2} \right) \\
        &\leq C \left( \| \mu \mathbf{1}_{\Omega} - \mu_{\infty} \|_{H^{-s}(\Omega)}^{2} + \mu \left( \mathbf{R}^{d} \setminus \Omega \right) \right),
    \end{split}
\end{equation}
where $C$ depends on $s,\alpha$, and $\Omega$. 

By equation \eqref{localization2} and property $8$ of Riesz-type gases (Definition \ref{def:riesztype}), we have that 
\begin{equation}
    \| \mu \mathbf{1}_{\Omega} - \mu_{\infty} \|_{H^{-s}(\Omega)}^{2} \leq C \mathcal{E}\left( \mu - \mu_{\infty} \right),
\end{equation}
where $C$ depends on $s$, $\Omega$, and $g$. 

On the other hand, by equation \eqref{lowerboundonxi}, we have that 
\begin{equation}
    \mu \left( \mathbf{R}^{d} \setminus \Omega \right) \leq \int_{\mathbf{R}^{d}} \zeta_{\infty} d\mu,
\end{equation}
where 
\begin{equation}
    \zeta_{\infty} =  V-2 h^{\mu_{\infty}}. 
\end{equation}

Putting everything together, and using the splitting formula, equation \eqref{eq:splittingform}, we have that
\begin{equation}
    \begin{split}
         \| \mu-\mu_{\infty} \|_{{C}^{0,\alpha}_{*}}^{2}  &\leq C_{\alpha} \left( \mathcal{E}\left( \mu - \mu_{\infty} \right) + \int_{\mathbf{R}^{d}} \zeta_{\infty} d\mu \right)\\
          &= C_{\alpha} \left( \mathcal{E}_{V}\left( \mu  \right) - \mathcal{E}_{V}\left( \mu_{\infty}  \right) \right),
    \end{split}
\end{equation}
where $C_{\alpha}$ depends only on $V,s$, $\alpha$, and $g$.
\end{proof}

We now prove the last item of Theorem \ref{T:transportineq}, restated here for convenience.

\begin{theorem}
Let $g$ be a Riesz-type kernel of order $s$, let $V$ be an admissible potential, and let $\alpha > s$.

\begin{itemize}
    \item[3.] Assume that 
\begin{equation}
    \liminf_{x \to \infty} \frac{V(x)}{|x|^{2 \alpha}} > 0.
\end{equation}
Then there exists a constant $C_{\alpha}$, which depends only on $V,g$, and $\alpha$  such that for every $\mu \in \mathcal{P}(\mathbf{R}^{d}),$ we have
\begin{equation}
   \| \mu-\mu_{\infty} \|_{\dot{C}^{0,\alpha}_{*}}^{2}  \leq C_{\alpha} \left( \mathcal{E}_{V}(\mu) - \mathcal{E}_{V}(\mu_{\infty}) \right).
\end{equation}
\end{itemize}
\end{theorem}

\begin{proof}
Since $\mu - \mu_{\infty}$ is nonnegative outside of $\Sigma,$ by Proposition \ref{localizingH-snorm}, there exists $\Omega$ containing $\Sigma$ such that 
\begin{equation}
    \parallel \mu - \mu_{\infty} \parallel_{H^{-s}} \geq C \parallel ( \mu - \mu_{\infty})\mathbf{1}_{\Omega} \parallel_{H^{-s}(\Omega)},
\end{equation}
where $C$ and $\Omega$ depends on $s$ and $V$. 
Let 
\begin{equation}
    \gamma :=   \liminf_{x \to \infty} \frac{V(x)}{|x|^{2 \alpha}}.
\end{equation}
We may also chose $\Omega$ such that 
\begin{equation}\label{lowerboundonxi2}
    \zeta_{\infty}(x) \geq \gamma \left( |x|^{\alpha} -1 \right)^{2}
\end{equation}
for $x$ outside of $\Omega.$

We then have that
\begin{equation}
\label{eq:eqforphi}
    \begin{split}
       \| \mu-\mu_{\infty} \|_{\dot{C}^{0,\alpha}_{*}}^{2}  &= \left( \sup_{| \phi |_{\dot{C}^{0,\alpha}}\leq 1} \int_{\mathbf{R}^{d}} \phi \, d(\mu - \mu_{\infty}) \right)^{2} \\
        &= \left( \sup_{| \phi |_{\dot{C}^{0,\alpha}}\leq 1, \  \phi(0)=-1} \int_{\mathbf{R}^{d}} \phi \, d(\mu - \mu_{\infty}) \right)^{2} \\
        &\leq 2 \left( \sup_{| \phi |_{\dot{C}^{0,\alpha}}\leq 1, \  \phi(0)=-1} \int_{\Omega} \phi \, d(\mu - \mu_{\infty}) \right)^{2} + 2\left( \sup_{| \phi |_{\dot{C}^{0,\alpha}}\leq 1, \  \phi(0)=-1} \int_{\mathbf{R}^{d} \setminus \Omega} \phi \, d\mu  \right)^{2}. 
    \end{split}
\end{equation}

We will deal with the first term now. Since $\Omega$ is compact, there exists $R>0$ such that
\begin{equation}
    \Omega \subset B(0,R).
\end{equation}
Then by Proposition \ref{localizingH-snorm}, we have that
\begin{equation}
    \begin{split}
        \sup_{| \phi |_{\dot{C}^{0,\alpha}}\leq 1, \  \phi(0)=-1} \int_{\Omega} \phi \, d(\mu - \mu_{\infty}) &\leq \sup_{| \phi |_{\dot{C}^{0,\alpha}}\leq 1, \  \|\phi \|_{L^{\infty}} \leq 1+R} \int_{\Omega} \phi \, d(\mu - \mu_{\infty}) \\
        &\leq C \sup_{| \phi |_{\dot{H}^{s}}\leq 1, \  \|\phi \|_{L^{2}} \leq 1} \int_{\Omega} \phi \, d(\mu - \mu_{\infty}) \\
        &\leq C \| (\mu - \mu_{\infty})\mathbf{1}_{\Omega} \|_{H^{-s}(\Omega)} \\
        &\leq C \sqrt{\mathcal{E}\left( \mu -\mu_{\infty}\right)},
    \end{split}
\end{equation}
where $C$ depends only on $\Omega, \alpha$, $s$, and $g$.

We now deal with the second term in the last line of equation \eqref{eq:eqforphi}. It is easy to see that the supremum is achieved at $\phi=\phi^{*}$, where 
\begin{equation}
     \phi^{*} := |x|^{\alpha}-1.
\end{equation}
Using Jensen's inequality, and the fact that $\mu \left( \mathbf{R}^{d} \setminus \Omega \right) \leq 1,$ we have that 
\begin{equation}
    \begin{split}
       \left( \sup_{| \phi |_{\dot{C}^{0,\alpha}}\leq 1, \  \phi(0)=-1} \int_{\mathbf{R}^{d} \setminus \Omega} \phi \, d\mu \right)^{2} &=  \left( \int_{\mathbf{R}^{d} \setminus \Omega} \phi^{*} \, d\mu \right)^{2} \\
       &\leq  \int_{\mathbf{R}^{d} \setminus \Omega} \left(\phi^{*}\right)^{2} \, d\mu. 
    \end{split}
\end{equation}
Using equation \eqref{lowerboundonxi2}, we have that 
\begin{equation}
    \begin{split}
        \left( \sup_{| \phi |_{\dot{C}^{0,\alpha}}\leq 1, \  \phi(0)=-1} \int_{\mathbf{R}^{d} \setminus \Omega} \phi \, d\mu \right)^{2} &\leq \int_{\mathbf{R}^{d} \setminus \Omega} \left(\phi^{*}\right)^{2} \, d\mu\\
        &\leq \frac{1}{\gamma} \int_{\mathbf{R}^{d}} \zeta_{\infty} \, d \mu.
    \end{split}
\end{equation}

Putting everything together, and using the splitting formula-equation \eqref{eq:splittingform} we have that
\begin{equation}
    \begin{split}
       \| \mu-\mu_{\infty} \|_{\dot{C}^{0,\alpha}_{*}}^{2}  &\leq C_{\alpha} \left( \mathcal{E}\left( \mu - \mu_{\infty} \right) + \int_{\mathbf{R}^{d}} \zeta_{\infty} d\mu \right)\\
          &= C_{\alpha} \left( \mathcal{E}_{V}\left( \mu  \right) - \mathcal{E}_{V}\left( \mu_{\infty}  \right) \right),
    \end{split}
\end{equation}
where $C_{\alpha}$ depends on $V, \alpha$, $s$, and $g$. 
\end{proof}

\begin{remark}
The same proof could be applied to the Coulomb kernel instead of a Riesz-type kernel, thus recovering Theorems 1.1 and 1.2 of \cite{chafai2018concentration}.
\end{remark}

We now apply Theorem \ref{T:transportineq} to prove Corollary \ref{Cor:cor}, restated here. 

\begin{corollary}
Let $V$ be an admissible potential, let $g$ be a Riesz-type kernel of order $s$, let $\alpha > s$. Let $\theta_{n}$ be a strictly positive sequence tending to infinity, and let
\begin{equation}
    \underline{\theta} := \min_{n} \theta_{n}.
\end{equation}

Then there exists a constant $C_{\alpha}$, which depends only on $V,g$, and $\alpha$ such that
\begin{equation}
    \| \mu_{\infty} - \mu_{\theta_{n}} \|_{{C}^{0,\alpha}_{*}}^{2} \leq \frac{C_{\alpha}}{\theta_{n}} \left| {\rm ent}[\mu_{\infty}] -  {\rm ent}[\mu_{ \underline{\theta}}] \right|.
\end{equation}

Additionally, if
\begin{equation}
    \liminf_{x \to \infty} \frac{V(x)}{|x|^{2\alpha}} > 0
\end{equation}
then there exits a constant $C_{\alpha}$, which depends only on $V,g$, and $\alpha$ such that
\begin{equation}
    \| \mu_{\infty} - \mu_{\theta_{n}} \|_{\dot{C}^{0,\alpha}_{*}}^{2} \leq \frac{C_{\alpha}}{\theta_{n}} \left| {\rm ent}[\mu_{\infty}] -  {\rm ent}[\mu_{ \underline{\theta}}] \right|.
\end{equation}
\end{corollary}

\begin{proof}
First, note that if $\theta_{p} < \theta_{q}$, then
\begin{equation}
    {\rm ent}[\mu_{\theta_{q}}] \geq  {\rm ent}[\mu_{\theta_{p}}]. 
\end{equation}
In particular,
\begin{equation}
    {\rm ent}[\mu_{\infty}] \geq  {\rm ent}[\mu_{\theta}]
\end{equation}
for any $\theta<\infty$.

We then have that, for any $n \in \mathbf{N}$
\begin{equation}
    \begin{split}
        \mathcal{E}_{V}(\mu_{\theta_{n}}) &=  \mathcal{E}_{\theta_{n}}(\mu_{\theta_{n}}) - \frac{1}{\theta_{n}} {\rm ent}[\mu_{\theta_{n}}]\\
         &\leq  \mathcal{E}_{\theta_{n}}(\mu_{\theta_{n}}) - \frac{1}{\theta_{n}} {\rm ent}[\mu_{\underline{\theta}}]\\
         &\leq  \mathcal{E}_{\theta_{n}}(\mu_{\infty}) - \frac{1}{\theta_{n}} {\rm ent}[\mu_{\underline{\theta}}]\\
         &=  \mathcal{E}_{V}(\mu_{\infty}) + \frac{1}{\theta_{n}} \left(  {\rm ent}[\mu_{\infty}]- {\rm ent}[\mu_{\underline{\theta}}]\right).
    \end{split}
\end{equation}

The conclusion follows from items 2 and 3 of Theorem \ref{T:transportineq}. 
\end{proof}

\section{Proof of concentration inequalities}
\label{sect:concentration}

In this section, we will prove Theorem \ref{T:Concentration}. The proof will be a consequence of Proposition \ref{localizingH-snorm}, as well as some preliminary results which we now state. 

\begin{proposition}
\label{prop:energydist}
Let $g$ be a Riesz-type kernel of order $s$, let $\mu$ have compact support and let $\alpha > s$. Then there exists a constant $C$, which depends on $V,g, {\rm supp}\mu$, and $\alpha$ such that
\begin{equation}\label{eq:compineq2}
    \| {\rm emp}_{N} -\mu \|_{{C}^{0,\alpha}_{*}} \leq N^{- \frac{\alpha}{d}} + C \left( {\rm F}_{N}(X_{N}, \mu) + C \|\mu\|_{L^{\infty}} N^{-\frac{2s}{d}} \right)^{\frac{1}{2}}.
\end{equation}
\end{proposition}
The proof is found in Secction~\ref{sect:appendixB}.

We need one more lemma. 

\begin{lemma}
\label{L:locprob}
Let $\mu, \nu \in \mathcal{P}(\mathbf{R}^{d})$ and let $\Omega \subset \mathbf{R}^{d}$ be such that $\text{supp}(\mu) \subset \Omega.$ Then for any $\alpha \in (0,1]$,
\begin{equation}
    \| (\mu - \nu)\mathbf{1}_{\Omega} \|_{{C}^{0,\alpha}_{*}} \geq \frac{1}{2}  \| \mu - \nu \|_{{C}^{0,\alpha}_{*}}.
\end{equation}
\end{lemma}

\begin{proof}
Proceed by contradiction and assume that 
\begin{equation}
    \| (\mu - \nu)\mathbf{1}_{\Omega} \|_{{C}^{0,\alpha}_{*}} < \frac{1}{2}  \| \mu - \nu \|_{{C}^{0,\alpha}_{*}}.
\end{equation}

Then by triangle inequality, we have that
\begin{equation}
    \| (\mu - \nu)\mathbf{1}_{\mathbf{R}^{d} \setminus \Omega} \|_{{C}^{0,\alpha}_{*}} \geq \frac{1}{2}  \| \mu - \nu \|_{{C}^{0,\alpha}_{*}}.
\end{equation}

However, it is easy to see that 
\begin{equation}
      \| (\mu - \nu)\mathbf{1}_{\mathbf{R}^{d} \setminus \Omega} \|_{{C}^{0,\alpha}_{*}} = \nu\left( \mathbf{R}^{d} \setminus \Omega \right).
\end{equation}

This implies that 
\begin{equation}
    \nu\left( \Omega \right) \leq 1- \frac{1}{2}  \| \mu - \nu \|_{{C}^{0,\alpha}_{*}},
\end{equation}
and therefore
\begin{equation}
    \begin{split}
        \| (\mu - \nu)\mathbf{1}_{\Omega} \|_{{C}^{0,\alpha}_{*}} &\geq \left| \mu\left( \Omega \right) - \nu\left( \Omega \right) \right| \\
        &\geq \frac{1}{2}  \| \mu - \nu \|_{{C}^{0,\alpha}_{*}}. 
    \end{split}
\end{equation}

This is a contradiction and therefore 
\begin{equation}
    \| (\mu - \nu)\mathbf{1}_{\Omega} \|_{{C}^{0,\alpha}_{*}} \geq \frac{1}{2}  \| \mu - \nu \|_{{C}^{0,\alpha}_{*}}.
\end{equation}
\end{proof}

We now prove the first item of Theorem \ref{T:Concentration}, restated here.

\begin{theorem}
Assume $V$ is an admissible potential, let $g$ be a Riesz-type kernel of order $s$, and let $\alpha > s$. Assume that $N \beta \to \infty$ and let $r > 0$. Then
\begin{itemize}

    \item[1.] There exists a constant $C$, which depends only on $V,g$, and $\alpha$ such that 
\begin{equation}
\begin{split}
    &\mathbf{P}_{N, \beta} \left( \| {\rm emp}_{N} - \mu_{\infty} \|_{{C}^{0,\alpha}_{*}} > r \right) \leq \\
    &\exp \Big( - N^{2} \beta \left( C[r-N^{-\frac{\alpha}{d}}]_{+}^{2} - C \|\mu_{\infty} \|_{L^{\infty}} N^{-\frac{2s}{d}} \right) + \\
    & \quad N \left( -\log \left| \Sigma \right| + \rm{ent}[\mu_{\infty}] + \beta \mathcal{E}(\mu_{\infty}) \right) + o(N) \Big),
\end{split}
\end{equation}
where $(\cdot)_{+}$ denotes the positive-part function. 
\end{itemize}
\end{theorem}

\begin{proof}
With the help of Lemmas \ref{lem:partfunc} and \ref{Prop:splitting} we can write
\begin{equation}
    \begin{split}
        &\mathbf{P}_{N, \beta} \left( \| {\rm emp}_{N} - \mu_{\infty} \|_{{C}^{0,\alpha}_{*}} > r \right) =\\ 
        &  \frac{1}{Z_{N, \beta}} \int_{\| {\rm emp}_{N} - \mu_{\infty} \|_{{C}^{0,\alpha}_{*}} > r } \exp \left( - \beta \mathcal{H}_{N}(X_{N}) \right) \, d X_{N} = \\
        & \frac{1}{Z_{N, \beta}} \int_{\| {\rm emp}_{N} - \mu_{\infty} \|_{{C}^{0,\alpha}_{*}}> r } \exp \left( - N^{2} \beta \left( \mathcal{E}_{V}(\mu_{\infty}) + {\rm{F}}_{N}(X_{N}, \mu_{\infty}) + \frac{1}{N} \sum_{i=1}^{N} \zeta_{\infty}(x_{i}) \right) \right) \, d X_{N} \leq \\
        & \left( \int_{\mathbf{R}^{d}} \exp \left( - N \beta  \zeta_{\infty}(x_{i}) \right) \, d X_{N} \right)^{-N} \cdot \\
        & \quad \min_{\| {\rm emp}_{N} - \mu_{\infty} \|_{{C}^{0,\alpha}_{*}} > r } \exp \left( - N^{2} \beta  {\rm{F}}_{N}(X_{N}, \mu_{\infty}) - N \beta \mathcal{E}(\mu_{\infty}) + N{\rm ent}[\mu_{\infty}] \right). 
    \end{split}
\end{equation}

The hypothesis $N \beta \to \infty$, implies that 
\begin{equation}
    \left( \int_{\mathbf{R}^{d}} \exp \left( - N \beta  \zeta_{\infty}(x_{i}) \right) \, d X_{N} \right)^{-N} = \exp \left( -N \log \left| \Sigma \right| + o(N) \right),
\end{equation}
since $N \beta  \zeta_{\infty}(x_{i}) = 0$ for $x_{i} \in \Sigma$ and $N \beta  \zeta_{\infty}(x_{i}) \to \infty$ for $x_{i} \notin \Sigma$.

On the other hand, by Proposition \ref{prop:energydist}, there exists a constant $C$ depending only on $V$ such that
\begin{equation}
    \begin{split}
        &\min_{\| {\rm emp}_{N} - \mu_{\infty} \|_{{C}^{0,\alpha}_{*}} > r } \exp \left( - N^{2} \beta  {\rm{F}}_{N}(X_{N}, \mu_{\infty}) - N \beta \mathcal{E}(\mu_{\infty}) + N{\rm ent}[\mu_{\infty}] \right) =\\
        &\exp \left( - N^{2} \beta  \min_{\| {\rm emp}_{N} - \mu_{\infty} \|_{{C}^{0,\alpha}_{*}} > r } \left \{ {\rm{F}}_{N}(X_{N}, \mu_{\infty}) \right\} - N \beta \mathcal{E}(\mu_{\infty}) + N{\rm ent}[\mu_{\infty}] \right) \leq \\
        &  \exp \left( - N^{2} \beta \left( C[r-N^{-\frac{\alpha}{d}}]_{+}^{2} - C \| \mu_{\infty} \|_{L^{\infty}} N^{-\frac{2s}{d}} \right) - N \beta \mathcal{E}(\mu_{\infty}) + N{\rm ent}[\mu_{\infty}] \right).
    \end{split}
\end{equation}

Putting everything together, we have that
\begin{equation}
\begin{split}
    &\mathbf{P}_{N, \beta} \left( \| {\rm emp}_{N} - \mu_{\infty} \|_{{C}^{0,\alpha}_{*}} > r \right) \leq \\
    &\exp \left( - N^{2} \beta \left( C[r-N^{-\frac{\alpha}{d}}]_{+}^{2} - C \|\mu_{\infty} \|_{L^{\infty}} N^{-\frac{2s}{d}} \right) + N \left( -\log \left| \Sigma \right| + \rm{ent}[\mu_{\infty}] + \beta \mathcal{E}(\mu_{\infty}) \right) + o(N) \right).
\end{split}
\end{equation}

\end{proof}

We proceed to the proof of the last item of Theorem \ref{T:Concentration}, which we restate here.

\begin{theorem}
Assume $V$ is an admissible potential, let $g$ be a Riesz-type kernel of order $s$, and let $\alpha > s$. Assume that $N \beta \to \infty$ and let $r > 0$.
\begin{itemize}
    \item[2.] Set $\theta = N \beta$, then there exists a constant $C$, which depends only on $V,g$, and $\alpha$ such that 
\begin{equation}
    \mathbf{P}_{N, \beta} \left( \| {\rm emp}_{N} - \mu_{\theta} \|_{{C}^{0,\alpha}_{*}} > r \right) \leq \exp \left( -N^{2} \beta \left( C\left[ r - C N^{-\frac{\alpha}{d}} \right]_{+}^{2} - C N^{-\frac{2s}{d}} \right) \right).
\end{equation}
\end{itemize}
\end{theorem}

\begin{proof}
We begin by using the splitting formula (equation \eqref{eq:thermspltfrm}) and Lemma \ref{lem:part} to write
\begin{equation}
\begin{split}
    \mathbf{P}_{N, \beta} \left( \| {\rm emp}_{N} - \mu_{\theta} \|_{{C}^{0,\alpha}_{*}} > r \right) &= \frac{1}{K_{N, \beta}} \int_{\| {\rm emp}_{N} - \mu_{\theta} \|_{{C}^{0,\alpha}_{*}} > r} \exp \left( -N^{2} \beta  {\rm{F}}_{N}(X_{N}, \mu_{\theta}) \right) \, d \mu_{\theta}(x_{i})\\
    &\leq \exp \left( -N^{2} \beta \min_{\| {\rm emp}_{N} - \mu_{\theta} \|_{{C}^{0,\alpha}_{*}} > r} {\rm{F}}_{N}(X_{N}, \mu_{\theta}) \right).
\end{split}    
\end{equation}

Let $\Omega$ be a compact set such that for $x \notin \Omega$,
\begin{equation}
\label{eq:expdecay2}
    \mu_{\theta}(x) \leq C \exp\left( -N \beta V(x) \right),
\end{equation}
for $C$ depending only on $V$ and $g$.

Let 
\begin{equation}
    \overline{\mu}_{\theta} = \frac{\mu_{\theta}\mathbf{1}_{\Omega} }{ \int_{\Omega} \mu_{\theta} \, \mathrm d x }.
\end{equation}

By Proposition \ref{prop:energydist}, since $ \overline{\mu}_{\theta}$ is a probability measure with compact support
\begin{equation}
    \| {\rm emp}_{N}-  \overline{\mu}_{\theta} \|_{{C}^{0,\alpha}_{*}} \leq N^{-\frac{\alpha}{d}} + C \left( {\rm{F}}_{N}(X_{N}, \overline{\mu}_{\theta}) + C N^{-\frac{2s}{d}} \right)^{\frac{1}{2}},
\end{equation}
for some $C$ depending only $V$ and $g$.

Using equation \eqref{eq:expdecay2}, we have
\begin{equation}
        \| {\rm emp}_{N}-  {\mu}_{\theta} \|_{{C}^{0,\alpha}_{*}} - {\rm err}_{N}  \leq N^{-\frac{\alpha}{d}} + C \left( {\rm{F}}_{N}(X_{N}, \mu_{\theta}) + {\rm err}_{N}  + C N^{-\frac{2s}{d}} \right)^{\frac{1}{2}},
\end{equation}
where
\begin{equation}
     {\rm err}_{N} = o(N^{p})
\end{equation}
for every $p \in \mathbf{R}$. 

This implies that
\begin{equation}
     \min_{\| {\rm emp}_{N} - \mu_{\theta} \|_{{C}^{0,\alpha}_{*}} > r} {\rm{F}}_{N}(X_{N}, \mu_{\theta}) \geq C\left[ r + {\rm err}_{N} - C N^{-\frac{\alpha}{d}} \right]_{+}^{2} + {\rm err}_{N} - C N^{-\frac{2s}{d}},
\end{equation}
and therefore 
\begin{equation}
    \mathbf{P}_{N, \beta} \left( \| {\rm emp}_{N} - \mu_{\theta} \|_{{C}^{0,\alpha}_{*}} > r \right) \leq \exp \left( -N^{2} \beta \left( C\left[ r + {\rm err}_{N} - C N^{-\frac{\alpha}{d}} \right]_{+}^{2} + {\rm err}_{N} - C N^{-\frac{2s}{d}} \right) \right),
\end{equation}
which is equivalent to the desired result by absorbing the ${\rm err}_{N}$ term into the constants. 
\end{proof}

\section{Lower bound for the distance}
\label{sect:optimality}

In this section, we prove Proposition \ref{Prop:opt}, which we restate here. 
\begin{proposition}
Let $\mu_{N}$ be a sequence of probability measures on $\mathbf{R}^{d}$ such that
\begin{equation}
    \sup_{N} \|\mu_{N}\|_{L^{\infty}} < \infty.
\end{equation}

Let $\nu_{N}$ be a sequence of probability measures on $\mathbf{R}^{d}$ such that
\begin{equation}
    \nu_{N} = \frac{1}{N} \sum_{i=1}^{N} \delta_{x_{i}^{N}}.
\end{equation}

Then
\begin{equation}
\label{eq:optimality}
    \|\mu_{N} - \nu_{N} \|_{{C}^{0,\alpha}_{*}} \geq C N^{-\frac{\alpha}{d}}, 
\end{equation}
where $C$ depends on $d$ and $\sup_{N} \|\mu_{N}\|_{L^{\infty}}$.
\end{proposition}

\begin{proof}
Let
\begin{equation}
    \overline{X}_{N} = \bigcup_{i=1}^{N} \{x_{i}^{N}\}.
\end{equation}

For $\lambda > 0$ to be determined later, define the function $\varphi_{\lambda}: \mathbf{R}^{d} \to \mathbf{R}$ as 
\begin{equation}
    \varphi_{\lambda}(x) = \left( \lambda N^{-\frac{\alpha}{d}} - \left[ {\rm dist}(x, \overline{X}_{N}) \right]^{\alpha} \right)_{+},
\end{equation}
where $(\cdot)_{+}$ denotes the positive part of $(\cdot)$. 

Note that for every $\lambda > 0$, the function $ \varphi_{\lambda}$ satisfies \begin{equation}
    \begin{split}
        |  \varphi_{\lambda} |_{\dot{C}^{0,\alpha}} &= 1 \\
        {\rm supp} (\varphi_{\lambda}) &= \bigcup_{i=1}^{N} B\left(x_{i}^{N}, \lambda N^{-\frac{1}{d}} \right).
    \end{split}
\end{equation}

Our aim is now to prove that 
\begin{equation}
    \int_{\mathbf{R}^{d}} \varphi_{\lambda} d \left( \nu_{N} - \mu_{N} \right) \geq k N^{-\frac{\alpha}{d}},
\end{equation}
for some $k \in \mathbf{R}^{+}$.

In order to do this, we introduce the notation
\begin{equation}
    M :=  \sup_{N} \|\mu_{N}\|_{L^{\infty}},
\end{equation}
and also the function
\begin{equation}
    \widetilde{\mu}_{N} = M \sum_{i=1}^{N} \mathbf{1}_{B(x_{i}^{N}, \lambda N^{-\frac{1}{d}})}. 
\end{equation}

Note that
\begin{equation}
     \int_{\mathbf{R}^{d}} \varphi_{\lambda} d \left( \nu_{N} - \mu_{N} \right) \geq \int_{\mathbf{R}^{d}} \varphi_{\lambda} d \left( \nu_{N} - \widetilde{\mu}_{N} \right),
\end{equation}
since 
\begin{equation}
    \begin{split}
         \int_{\mathbf{R}^{d}} \varphi_{\lambda} \,  d \mu_{N} &\leq M \sum_{i=1}^{N}  \int_{\mathbf{R}^{d}} \mathbf{1}_{B(x_{i}^{N}, \lambda N^{-\frac{1}{d}})}  \varphi_{\lambda} \, \mathrm d x \\
         &= \int_{\mathbf{R}^{d}} \varphi_{\lambda} \,  d  \widetilde{\mu}_{N}.
    \end{split}
\end{equation}

We now compute 
\begin{equation}
    \begin{split}
        \int_{\mathbf{R}^{d}} \varphi_{\lambda} d \left( \nu_{N} - \widetilde{\mu}_{N} \right) &=  \lambda N^{-\frac{\alpha}{d}} - \int_{\mathbf{R}^{d}} \varphi_{\lambda} \,  d  \widetilde{\mu}_{N}\\
        &\geq \lambda N^{-\frac{\alpha}{d}} - M \left( \sum_{i=1}^{N} \int_{B(x_{i}^{N}, \lambda N^{-\frac{1}{d}})} \lambda N^{-\frac{\alpha}{d}} \, \mathrm d x \right) \\
        &= \lambda N^{-\frac{\alpha}{d}} - \lambda M N^{-\frac{\alpha}{d}} \left( N k_{d} [\lambda N^{-\frac{1}{d}}]^{d} \right) \\
        &= \lambda N^{-\frac{\alpha}{d}}  \left( 1 - Mk_{d} \lambda^{d} \right),
    \end{split}
\end{equation}
where $k_{d}$ is the volume of the $d-$dimensional unit sphere. Taking
\begin{equation}
    \lambda = \frac{1}{\left(2 M k_{d} \right)^{\frac{1}{d}}},
\end{equation}
we have that 
\begin{equation}
     \int_{\mathbf{R}^{d}} \varphi_{\lambda} d \left( \nu_{N} - \widetilde{\mu}_{N} \right) \geq N^{-\frac{\alpha}{d}} \left( \frac{1}{2 ^{\frac{d+1}{d}}  \left(M k_{d} \right)^{\frac{1}{d}}} \right),
\end{equation}
which implies equation \eqref{eq:optimality}.

\end{proof}

\section{Bound on the Laplace transform}
\label{sect:MTineq}

This section is devoted to proving Theorem \ref{theo:MTineq}. The proof will rely on a concentration result from \cite{padilla2020concentration}:
\begin{theorem}\label{rateofconvergencethermalequilibrium}
Let $d \geq 2$, assume that $\frac{1}{N} \ll \beta$ and that $V$ is admissible. Then there exists a constant $C$ depending only on $V$ such that for any $r>0,$ we have
\begin{equation}
    \mathbf{P}_{N, \beta} \left( \parallel  {\rm emp}_{N}-\mu_{\theta} \parallel_{BL} \leq\frac{r}{N^{\frac{1}{d}}} \right) \geq 1- \exp\left(- N^{2-\frac{2}{d}}\beta\left( C(r-C)_{+}^{2}-C \right) \right),
\end{equation}
\end{theorem}

We will also use a simple concentration inequality, which is an immediate consequence of the splitting formula (Proposition \ref{Prop:splitting}) and the fact that the log of the next-order partition function is positive (Lemma \ref{lem:part}). The proof is omitted. 
\begin{lemma}
\label{lem:fundconc}
For any $r>0$, 
\begin{equation}
    \mathbf{P}_{N, \beta} ({\rm F}_{N}(X_{N}, \mu_{\theta})> r) \leq \exp \left( -N^{2} \beta r \right). 
\end{equation}
\end{lemma}

We now prove Theorem \ref{theo:MTineq}, restated here,

\begin{theorem}
\label{theo:MTineq2}
Let $V$ be an admissible potential, and assume $N \beta \to \infty$. Let $f:\mathbf{R}^{d} \to \mathbf{R}$ be continuous, and define the random variable ${\rm Fluct}[f]$ by 
\begin{equation}
    {\rm Fluct}[f] = \int_{\mathbf{R}^{d}} f d\left( {\rm emp}_{N} - \mu_{\theta} \right).
\end{equation}

\begin{itemize}
    \item[1.]If $g$ is the Coulomb kernel, then there exist a constant $C$, which depends only on $V$ such that
\begin{equation}
    \log \left( \mathbf{E}_{\mathbf{P}_{N, \beta}} \exp \left( N^{2} \beta \left| t{\rm Fluct}[f] \right| \right) \right) \leq N^{2} \beta \left( C t^{2}\|f\|_{W^{1, \infty}}^{2} + N^{-\frac{2}{d}} C \right).
\end{equation}

    \item[2. ]If $g$ is a Riesz-type kernel of order $s$ then for any $\alpha > s$ there exists a constant $C$, which depends only on $V$, $g$, and $\alpha$ such that
\begin{equation}
    \log \left( \mathbf{E}_{\mathbf{P}_{N, \beta}} \exp \left( N^{2} \beta \left| t{\rm Fluct}[f] \right| \right) \right) \leq  N^{2} \beta \left( C t^{2}\|f\|_{{C}^{0,\alpha}}^{2} + N^{-\frac{2 \alpha}{d}} C \right).
\end{equation}

    \item[3.]If $g$ is the Coulomb kernel, $i \in \{0,1\}$ and $\alpha \in (0,1)$ then there exists a constant $C$, which depends only on $V$ such that
\begin{equation}
\begin{split}
    &\log \left( \mathbf{E}_{\mathbf{P}_{N, \beta}} \exp \left( N^{2} \beta \left| t{\rm Fluct}[f] \right| \right) \right) \leq\\
    &\beta N^{2} \left( \frac{1}{4} t^{2}|f|_{\dot{H}^{1}}^{2} + N^{- \frac{i+\alpha}{d}} t  |f|_{\dot{C}^{i, \alpha}} + C N^{-\frac{2}{d}} \right)+ \frac{d}{2}\left( \log (N^{2} \beta) + \log (1+ t|f|_{\dot{H}^{1}}) \right). 
\end{split}    
\end{equation}

    \item[4.]If $g$ is a Riesz-type kernel of order $s$, $i \in \{0,1\}$ and $\alpha \in (0,1)$, then there exists a constant $C$ depending only on $V$ and $g$ such that
\begin{equation}
    \log \left( \mathbf{E}_{\mathbf{P}_{N, \beta}} \exp \left( N^{2} \beta \left| t{\rm Fluct}[f] \right| \right) \right) \leq \beta N^{2} \left( C t^{2} |f|_{\dot{H}^{s}}^{2} + C N^{- \frac{i+\alpha}{d}} t |f|_{\dot{C}^{i, \alpha}} + C N^{-\frac{2s}{d}} \right). 
\end{equation}
\end{itemize}
\end{theorem}

\begin{proof}

\textbf{Step 1:} Proof of item $1.$

We start with the case of $g$ given by the Coulomb kernel. Since $\mathbf{E}_{\mathbf{P}_{N, \beta}} \exp \left( N^{2} \beta \left| t{\rm Fluct}[f] \right| \right)$ is a non-negative random variable, we can rewrite the expectation as 
\begin{equation}
    \begin{split}
           \mathbf{E}_{\mathbf{P}_{N, \beta}} \exp \left( N^{2} \beta \left| t{\rm Fluct}[f] \right| \right) &= \int_{0}^{\infty} \mathbf{P}_{N, \beta} \left( \exp \left( N^{2} \beta \left| t{\rm Fluct}[f] \right| \right) > x \right) \, \mathrm d x\\
         &= \int_{0}^{\infty} \mathbf{P}_{N, \beta}  \left( t{\rm Fluct}[f]  > \frac{ \log x}{ N^{2} \beta} \right) \, \mathrm d x\\
         & \leq \int_{0}^{\infty} \mathbf{P}_{N, \beta}  \left( t \|f\|_{W^{1, \infty}} \|{\rm emp}_{N} - \mu_{\theta}\|_{\rm BL}  > \frac{ \log x}{ N^{2} \beta} \right) \, \mathrm d x.
    \end{split}
\end{equation}

Note that if $x \leq 1$, then 
\begin{equation}
    \mathbf{P}_{N, \beta}  \left( t \|f\|_{W^{1, \infty}} \|{\rm emp}_{N} - \mu_{\theta}\|_{\rm BL}  > \frac{ \log x}{ N^{2} \beta} \right) = 1. 
\end{equation}
On the other hand, if $x > 1$ then by Theorem \ref{rateofconvergencethermalequilibrium}
\begin{equation}
    \mathbf{P}_{N, \beta}  \left( t \|f\|_{W^{1, \infty}} \|{\rm emp}_{N} - \mu_{\theta}\|_{\rm BL}  > \frac{ \log x}{ N^{2} \beta} \right) \leq \exp \left( - N^{2} \beta \left[ C \left( \frac{\log x}{N^{2} \beta t \|f\|_{W^{1, \infty}}} \right)^{2} - C N^{-\frac{2}{d}} \right] \right).
\end{equation}

Therefore we infer that 
\begin{equation}
    \begin{split}
        &  \mathbf{E}_{\mathbf{P}_{N, \beta}} \exp \left( N^{2} \beta \left| t{\rm Fluct}[f] \right| \right)  \\
         \leq & 1 + \exp \left( C N^{2 - \frac{2}{d}} \beta \right) \int_{0}^{\infty} \exp \left( - \left[ C  \frac{[\log x]^{2}}{N^{2} \beta t^{2} \|f\|^{2}_{W^{1, \infty}}} \right] \right) \, \mathrm d x.
    \end{split}
\end{equation}

Performing the change of variables $y = \log x$ we can transform the integral into 
\begin{equation}
    \begin{split}
        &  \mathbf{E}_{\mathbf{P}_{N, \beta}} \exp \left( N^{2} \beta \left| t{\rm Fluct}[f] \right| \right)  \\
         \leq & 1 + \exp \left( C N^{2 - \frac{2}{d}} \beta \right) \int_{-\infty}^{\infty} \exp \left( - \left[ C  \frac{y^{2}}{N^{2} \beta t^{2} \|f\|^{2}_{W^{1, \infty}}} - y\right] \right) \, \mathrm d y.
    \end{split}
\end{equation}

The last integral is a Gaussian density that can be computed exactly by completing squares: 
\begin{equation}
    \begin{split}
        & \int_{-\infty}^{\infty} \exp \left( - \left[ C  \frac{y^{2}}{N^{2} \beta t^{2} \|f\|^{2}_{W^{1, \infty}}} - y\right] \right) \, \mathrm d y\\
        =  & \int_{-\infty}^{\infty} \exp \left( - \frac{C}{N^{2} \beta t^{2} \|f\|^{2}_{W^{1, \infty}}} \left( y - \frac{ N^{2} \beta t^{2} \|f\|^{2}_{W^{1, \infty}}}{2C} \right)^{2} +  \frac{ N^{2} \beta t^{2} \|f\|^{2}_{W^{1, \infty}}}{4C} \right) \, \mathrm d y \\
        =& \exp \left( \frac{ N^{2} \beta t^{2} \|f\|^{2}_{W^{1, \infty}}}{4C} \right) \left( \frac{ N^{2} \beta t^{2} \|f\|^{2}_{W^{1, \infty}}}{2C} \right)^{\frac{d}{2}}
    \end{split}
\end{equation}

We conclude that 
\begin{equation}
    \begin{split}
        & \log \left( \mathbf{E}_{\mathbf{P}_{N, \beta}} \exp \left( N^{2} \beta \left| t{\rm Fluct}[f] \right| \right) \right) \\
         \leq & \log \left( 1 + \exp \left( C N^{2 - \frac{2}{d}} \beta \right)\exp \left( \frac{ N^{2} \beta t^{2} \|f\|^{2}_{W^{1, \infty}}}{4C} \right) \left( \frac{ N^{2} \beta t^{2} \|f\|^{2}_{W^{1, \infty}}}{2C} \right)^{\frac{d}{2}} \right)\\
         \leq  &N^{2} \beta \left( C t^{2}\|f\|_{W^{1, \infty}}^{2} + N^{-\frac{2}{d}} C \right).
    \end{split}
\end{equation}

\textbf{Step 2:} Proof of item $2.$

The proof of the statement in the case of Riesz-type kernels is the same, after noting that if $g$ is a Riesz-type kernel of order $s$, then for any $\alpha > s$ there exist a constant $C$ depending on $V,d, g$ and $\alpha$ such that 
\begin{equation}
    \mathbf{P}_{N, \beta} \left( \parallel  {{\rm emp}}_{N}-\mu_{\theta} \parallel_{C^{0,\alpha}_{*}} \geq r \right) \leq \exp\left(- N^{2}\beta\left( C r^{2} - C N^{-\frac{2 \alpha}{d}} \right) \right).
\end{equation}

\textbf{Step 3:} Proof of item $3.$

We now turn to proving the third item. 

We introduce the notation 
\begin{equation}
    {\rm emp}_{N}^{*} =  {\rm emp}_{N} \ast \delta_{N^{-\frac{1}{d}}},
\end{equation}
where $\delta_{N^{-\frac{1}{d}}}$ denotes the uniform probability measure on $B(0,N^{-\frac{1}{d}}),$ and 
\begin{equation}
    {\rm Fluct}^{*}[f]  = \int_{\mathbf{R}^{d}} f \,  d ({\rm emp}_{N}^{*} - \mu_{\theta}).
\end{equation}

Note that if $i=0$, then
\begin{equation}
    \begin{split}
        \left| \int_{\mathbf{R}^{d}} f \, d({\rm emp}_{N}^{*} - {\rm emp}_{N}) \right| &\leq \frac{1}{N} \sum_{j=1}^{N} \left| f(x_{j}) -  \fint_{B(0, N^{-\frac{1}{d}})} f(x_{j}) + (f(y) - f(x_{j})) \, \mathrm d y \right| \\
        &\leq N^{-\frac{\alpha}{d}} |f|_{\dot{C}^{0,\alpha}}. 
    \end{split}
\end{equation}

On the other hand, for $i=1$, we have
\begin{equation}
    \begin{split}
        &\left| \int_{\mathbf{R}^{d}} f \, d({\rm emp}_{N}^{*} - {\rm emp}_{N}) \right| \leq \\
        &\frac{1}{N} \sum_{j=1}^{N} \left| f(x_{j}) -  \fint_{B(0, N^{-\frac{1}{d}})} f(x_{j}) + \nabla f_{x_{j}} \cdot (y-x_{j})  - (f(y) - f(x_{j}) + \nabla f_{x_{j}} \cdot (y-x_{j})) \, \mathrm d y \right| \leq \\
        & N^{-\frac{1+\alpha}{d}} |f|_{\dot{C}^{1,\alpha}}. 
    \end{split}
\end{equation}

In either case, we get that
\begin{equation}
    \left| {\rm Fluct}^{*}[f] - {\rm Fluct}[f]  \right| \leq N^{-\frac{i+\alpha}{d}} |f|_{\dot{C}^{i,\alpha}}. 
\end{equation}

We now proceed as in steps $1$ and $2$, and write
\begin{equation}
    \begin{split}
           \mathbf{E}_{\mathbf{P}_{N, \beta}} \exp \left( N^{2} \beta \left| t{\rm Fluct}[f] \right| \right) &= \int_{0}^{\infty} \mathbf{P}_{N, \beta} \left( \exp \left( N^{2} \beta \left| t{\rm Fluct}[f] \right| \right) > x \right) \, \mathrm d x\\
         &= \int_{0}^{\infty} \mathbf{P}_{N, \beta}  \left( t{\rm Fluct}[f]  > \frac{ \log x}{ N^{2} \beta} \right) \, \mathrm d x\\
         & \leq \int_{0}^{\infty} \mathbf{P}_{N, \beta}  \left( t{\rm Fluct}^{*}[f]  > \frac{ \log x}{ N^{2} \beta} - N^{-\frac{i + \alpha}{d}} |f|_{\dot{C}^{i,\alpha}} \right) \, \mathrm d x\\
         & \leq \int_{0}^{\infty} \mathbf{P}_{N, \beta}  \left( t|f|_{\dot{H}^{1}} \| {\rm emp}^{*}_{N} - \mu_{\theta} \|_{H^{-1}}  > \frac{ \log x}{ N^{2} \beta} - N^{-\frac{i + \alpha}{d}} t|f|_{\dot{C}^{i,\alpha}} \right) \, \mathrm d x\\
         & = \int_{0}^{\infty} \mathbf{P}_{N, \beta}  \left(  \| {\rm emp}^{*}_{N} - \mu_{\theta} \|_{H^{-1}}  > \frac {\frac{ \log x}{ N^{2} \beta} - N^{-\frac{i + \alpha}{d}} t|f|_{\dot{C}^{i,\alpha}}}{t|f|_{\dot{H}^{1}}} \right) \, \mathrm d x.
    \end{split}
\end{equation}

Note that if $\frac{ \log x}{ N^{2} \beta} - N^{-\frac{i + \alpha}{d}} t|f|_{\dot{C}^{i,\alpha}} < 0$, i.e. if $x < \exp \left( N^{2-\frac{i + \alpha}{d}} \beta t|f|_{\dot{C}^{i,\alpha}} \right)$, then
\begin{equation}
    \mathbf{P}_{N, \beta}  \left(  \| {\rm emp}_{N}^{*} - \mu_{\theta} \|_{H^{-1}}  > \frac {\frac{ \log x}{ N^{2} \beta} - N^{-\frac{i + \alpha}{d}} t|f|_{\dot{C}^{i,\alpha}}}{t|f|_{\dot{H}^{1}}} \right) =1.
\end{equation}

Otherwise, if $x \geq \exp \left( N^{2} \beta t|f|_{\dot{C}^{i,\alpha}} \right)$, then by Lemma \ref{lem:fundconc},
\begin{equation}
    \begin{split}
             &\mathbf{P}_{N, \beta}  \left(  \| {\rm emp}^{*}_{N} - \mu_{\theta} \|_{H^{-1}}  > \frac {\frac{ \log x}{ N^{2} \beta} - N^{-\frac{i + \alpha}{d}} t|f|_{\dot{C}^{i,\alpha}}}{t|f|_{\dot{H}^{1}}} \right)\\
            = &\mathbf{P}_{N, \beta}  \left(  \| {\rm emp}^{*}_{N} - \mu_{\theta} \|_{H^{-1}}^{2}  > \left[ \frac {\frac{ \log x}{ N^{2} \beta} - N^{-\frac{i + \alpha}{d}} t|f|_{\dot{C}^{i,\alpha}}}{t|f|_{\dot{H}^{1}}} \right]^{2} \right) \\
            \leq & \mathbf{P}_{N, \beta}  \left(  {\rm F}_{N}(X_{N}, \mu_{\theta})  > \left[ \frac {\frac{ \log x}{ N^{2} \beta} - N^{-\frac{i + \alpha}{d}} t|f|_{\dot{C}^{i,\alpha}}}{t|f|_{\dot{H}^{1}}} \right]^{2} - CN^{- \frac{2}{d}} \right)  \\
            \leq & \exp \left(  CN^{2- \frac{2}{d}} \beta \right) \exp \left( \left[ \frac {\frac{ \log x}{ N^{2} \beta} - N^{-\frac{i + \alpha}{d}} t|f|_{\dot{C}^{i,\alpha}}}{t|f|_{\dot{H}^{1}}} \right]^{2} \right),
    \end{split}
\end{equation}
for some $C$ depending only on $V$. 

Therefore,
\begin{equation}
    \begin{split}
         &\mathbf{E}_{\mathbf{P}_{N, \beta}} \exp \left( N^{2} \beta \left| t{\rm Fluct}[f] \right| \right) \\
         \leq & \exp \left( N^{2-\frac{i + \alpha}{d}} \beta t|f|_{\dot{C}^{i,\alpha}} \right) + \exp \left(  CN^{2- \frac{2}{d}} \beta \right) \int_{0}^{\infty} \exp \left( \left[ \frac {\frac{ \log x}{ N^{2} \beta} - N^{-\frac{i + \alpha}{d}} t|f|_{\dot{C}^{i,\alpha}}}{t|f|_{\dot{H}^{1}}} \right]^{2} \right) \, \mathrm d x.
    \end{split}
\end{equation}

Proceeding as in the previous step, we can perform the change of variables $y = \log x$, then complete squares, and use the formula for the integral of a Gaussian density, and get 
\begin{equation}
\begin{split}
    &\int_{0}^{\infty} \exp \left( \left[ \frac {\frac{ \log x}{ N^{2} \beta} - N^{-\frac{i + \alpha}{d}} t|f|_{\dot{C}^{i,\alpha}}}{t|f|_{\dot{H}^{1}}} \right]^{2} \right) \, \mathrm d x =\\
    &\left( \frac {N^{2} \beta t|f|_{\dot{H}^{1}}}{2} \right)^{\frac{d}{2}} \exp \left( N^{2} \beta \left[ \frac{t^{2}|f|_{\dot{H}^{1}}^{2}}{4} + N^{-\frac{i + \alpha}{d}} t|f|_{\dot{C}^{i,\alpha}} \right] \right).
\end{split}    
\end{equation}

Note that
\begin{equation}
    \begin{split}
        &\max \Bigg\{ N^{2-\frac{i + \alpha}{d}} \beta t|f|_{\dot{C}^{i,\alpha}}, \\
        & \quad \log \left( \exp \left(  CN^{2- \frac{2}{d}} \beta \right) \left(\frac {N^{2} \beta t|f|_{\dot{H}^{1}}}{2} \right)^{\frac{d}{2}} \exp \left( N^{2} \beta \left[ \frac{t^{2}|f|_{\dot{H}^{1}}^{2}}{4} + N^{-\frac{i + \alpha}{d}} t|f|_{\dot{C}^{i,\alpha}} \right] \right) \right)  \Bigg\} \leq\\
        &  CN^{2- \frac{2}{d}} \beta + \frac{d}{2}\left( \log (N^{2} \beta) + \log (1+ t|f|_{\dot{H}^{1}}) \right) + N^{2} \beta \left[ \frac{t^{2}|f|_{\dot{H}^{1}}^{2}}{4} + N^{-\frac{i + \alpha}{d}} t|f|_{\dot{C}^{i,\alpha}} \right].  
    \end{split}
\end{equation}

From this we can conclude. 

\textbf{Step 4:} Proof of item $4.$

We now turn to proving the last item of the Theorem. The procedure is similar to the proof of the last item, the difference is that the smearing procedure goes through higher dimensions. 

Let $\overline{f}: \mathbf{R}^{d+m} \to \mathbf{R}$ be such that $\overline{f} (x,0) = f(x) \forall x$ and 
\begin{equation}
    \begin{split}
       | \overline{f} |_{\dot{H}^{s + \frac{m}{2}}} &\leq C   | {f} |_{\dot{H}^{s}}\\
       | \overline{f} |_{\dot{C}^{0,\alpha}} &\leq C   | {f} |_{\dot{C}^{0,\alpha}},
    \end{split}
\end{equation}
for some constant depending only on $d$ and $m$. For example, take $\overline{f}(x,z) = \left( e^{|z| \Delta_{\mathbf{R}^{d}}} f \right)(x)$.

We introduce the notation $\delta_{x}^{(\eta)}$ for the uniform probability measure on $\partial B(x, \eta) \subset \mathbf{R}^{d+m}$, and define 
\begin{equation}
    {\rm emp}_{N}^{*} =   {\rm emp}_{N} \ast \delta_{x}^{(N^{-\frac{1}{d}})}, 
\end{equation}
and 
\begin{equation}
    {\rm Fluct}^{*}[f] = \int_{\mathbf{R}^{d + m}} \overline{f} \, d (\mu_{\theta} - {\rm emp}_{N}^{*} ). \end{equation}

Proceeding as in the proof of the last item, we have
\begin{equation}
    \begin{split}
       \left|  {\rm Fluct}^{*}[f] -  {\rm Fluct}[f] \right| & \leq  N^{- \frac{i+\alpha}{d}} |\overline{f} |_{\dot{C}^{i, \alpha}}\\
       & \leq C N^{- \frac{i+\alpha}{d}} |{f} |_{\dot{C}^{i, \alpha}}.
    \end{split}
\end{equation}

Proceeding as in step 3, we have that
\begin{equation}
    \begin{split}
           &\mathbf{E}_{\mathbf{P}_{N, \beta}} \exp \left( N^{2} \beta \left| t{\rm Fluct}[f] \right| \right) = \\
           &\int_{0}^{\infty} \mathbf{P}_{N, \beta} \left( \exp \left( N^{2} \beta \left| t{\rm Fluct}[f] \right| \right) > x \right) \, \mathrm d x =\\
         & \int_{0}^{\infty} \mathbf{P}_{N, \beta}  \left( t{\rm Fluct}[f]  > \frac{ \log x}{ N^{2} \beta} \right) \, \mathrm d x \leq\\
         & \int_{0}^{\infty} \mathbf{P}_{N, \beta}  \left( t{\rm Fluct}^{*}[f]  > \frac{ \log x}{ N^{2} \beta} - C N^{-\frac{i + \alpha}{d}} |f|_{\dot{C}^{i,\alpha}} \right) \, \mathrm d x \leq\\
         &  \int_{0}^{\infty} \mathbf{P}_{N, \beta}  \left( C t|\overline{f}|_{\dot{H}^{s + \frac{m}{2}}} \| {\rm emp}^{*}_{N} - \mu_{\theta} \|_{H^{-s - \frac{m}{2}}}  > \frac{ \log x}{ N^{2} \beta} - C N^{-\frac{i + \alpha}{d}} t|f|_{\dot{C}^{i,\alpha}} \right) \, \mathrm d x =\\
         & \int_{0}^{\infty} \mathbf{P}_{N, \beta}  \left( \| {\rm emp}^{*}_{N} - \mu_{\theta} \|_{H^{-s - \frac{m}{2}}} > \frac {\frac{ \log x}{ N^{2} \beta} - C N^{-\frac{i + \alpha}{d}} t|f|_{\dot{C}^{i,\alpha}}}{C t|f|_{\dot{H}^{s}}} \right) \, \mathrm d x.
    \end{split}
\end{equation}

Proceeding as in the previous step, we can split split the integral into the domains $\frac{ \log x}{ N^{2} \beta} - N^{-\frac{i + \alpha}{d}} t|f|_{\dot{C}^{i,\alpha}} >0$ and $\frac{ \log x}{ N^{2} \beta} - N^{-\frac{i + \alpha}{d}} t|f|_{\dot{C}^{i,\alpha}} <0$, then apply Lemma \ref{lem:RosSer} with $\eta_{i} = N^{-\frac{1}{d}}$ for each $i$, and Lemma \ref{lem:fundconc}, then bound/compute the integrals, and get 
\begin{equation}
    \begin{split}
        &\mathbf{E}_{\mathbf{P}_{N, \beta}} \exp \left( N^{2} \beta \left| t{\rm Fluct}[f] \right| \right) \leq  \\
        & \exp \left( N^{2-\frac{i + \alpha}{d}} \beta t|f|_{\dot{C}^{i,\alpha}} \right) + \\
        & \quad  \exp \left(  CN^{2- \frac{2s}{d}} \beta \right) \left( N^{2} \beta t|f|_{\dot{H}^{s}} \right)^{\frac{d}{2}} \exp \left( C N^{2} \beta \left[ {t^{2}|f|_{\dot{H}^{s}}^{2}} + N^{-\frac{i + \alpha}{d}} t|f|_{\dot{C}^{i,\alpha}} \right] \right).
    \end{split}
\end{equation}

From this we can conclude. 

\end{proof}

\section{Appendix A: the $H^{-s}(\Omega)$ norm}
\label{sect:appendixA}

This section is devoted to proving Proposition \ref{localizingH-snorm}. The strategy is an extension of the proof \cite{padilla2020concentration} for an analogous result. The proof will rely on a few lemmas. 

\begin{lemma}\label{extensionlemma}
Let $\Omega \subset \mathbf{R}^{d}$ be an open bounded set with a $C^{2}$ boundary, and let $f \in H^{s}({\Omega}).$ 
Let 
\begin{equation}
    \epsilon_{*}=\sup \{ \epsilon>0| x \mapsto x+\epsilon\widehat{n}(x) \text{ is a diffeomorphism for all } |\delta|<\epsilon \},
\end{equation}
where $\widehat{n}(x)$ is the unit normal to $\partial \Omega$ at $x.$ Then for every $0<\epsilon<\epsilon_{*}$ there exists $f_{\epsilon} \in \dot{H}^{s}(\mathbf{R}^{d})$ such that 
\begin{itemize}
    \item[1.]  $f_{\epsilon}|_{\Omega}=f.$
    \item[2.]  $\mbox{supp}(f_{\epsilon}) \subset \overline{\Omega}^{\epsilon},$ where
    \begin{equation}
        \Omega^{\epsilon}=\{ x \in \mathbf{R}^{d}| d(x,\Omega) < \epsilon\}.
    \end{equation}
    \item[3.] 
    \begin{equation}
        \begin{split}
            \parallel  f_{\epsilon} \parallel_{\dot{H}^{s}}^{2} &\leq \frac{C}{{\epsilon}^{\frac{s}{2}}} \parallel f \parallel_{H^{s}}\\
            \parallel f_{\epsilon} \parallel_{L^{2}} &\leq (1+C\sqrt{\epsilon}) \parallel f \parallel_{L^{2}},
        \end{split}
    \end{equation}
    where $C$ depends only on $\Omega$. 
\end{itemize}

In addition, if $f$ is non-negative in a neighborhood of $\partial \Omega,$ then $f_{\epsilon}$ is non negative in $\Omega^{\epsilon} \setminus \Omega.$
\end{lemma}

\begin{proof}
\textbf{Step 1:} Flat case.

Assume for this step that for some $\delta>0$
\begin{equation}
    B(x,\delta) \bigcap \partial \Omega = \{ y| y_{d}=0\} \bigcap B(x,\delta).
\end{equation}

We will use the notation 
\begin{equation}
    x=(\underline{x}, x_{d}),
\end{equation}
where $\underline{x} \in \mathbf{R}^{d-1}$ and $x_{d} \in \mathbf{R}.$ Let $x \in \partial \Omega$ and let
\begin{equation}
    \underline{B}(x,\delta)= \{ y| y_{d}=0\} \bigcap B(x,\delta).
\end{equation}
 Let $\alpha>0$ be such that 
\begin{equation}
     \underline{B}(x,\delta) \times (-\alpha, 0) \subset \Omega.
\end{equation}
Define $\varphi: \underline{B}(x,\delta) \times (0,\alpha) \to \mathbf{R}$ by odd reflection: 
\begin{equation}
    \varphi(\underline{y}, y_{d})= f(\underline{y}, -y_{d}).
\end{equation}
Let $\mu\in C^{\infty}( [0,\alpha], \mathbf{R}^{+})$ be such that $\mu(0)=1,\    \mu(\alpha)=0,$ and $\mu$ is decreasing. Consider now $\widehat{\varphi}: \underline{B}(x,\delta) \times (0,\alpha) \to \mathbf{R}$ defined as 
\begin{equation}
    \widehat{\varphi}(\underline{y}, y_{d})=\varphi(\underline{y}, y_{d}) \mu(y_{d}).
\end{equation}
Now we define the function $\varphi_{\epsilon}: \underline{B}(x,\delta) \times (0,\epsilon) \to \mathbf{R}$ as
\begin{equation}
    \varphi_{\epsilon}(\underline{y}, y_{d})=\widehat{\varphi}\left(\underline{y}, \frac{\alpha}{\epsilon}y_{d}\right).
\end{equation}
We immediately get the estimate
\begin{equation}
        \parallel \varphi_{\epsilon} \parallel_{L^{2}} \leq \sqrt{\frac{\epsilon}{\alpha}}  \parallel f \parallel_{L^{2}}.
\end{equation}
We also have the estimate
\begin{equation}
        \parallel \varphi_{\epsilon} \parallel_{\dot{H}^{s}} \leq C \max\left( \left(\frac{{\alpha}}{{\epsilon}} \right)^{\frac{s}{2}}, 1 \right)  \parallel f \parallel_{H^{1}},
\end{equation}
where $C$ depends only on $\Omega.$

Lastly, if $f$ is non-negative in a neighborhood of $\partial \Omega$ consider the function 
\begin{equation}
    M\varphi_{\epsilon}=\max \left( \varphi_{\epsilon}, 0 \right).
\end{equation}
Then $M\varphi_{\epsilon}$ is positive, and $M\varphi_{\epsilon}$ agrees with $f$ on $\{x_{d} = 0\}$.

We also have that that
\begin{equation}
    \begin{split}
        \parallel  M\varphi_{\epsilon} \parallel_{\dot{H}^{s}} &\leq  \parallel  \varphi_{\epsilon} \parallel_{\dot{H}^{s}} \\
        \parallel  M\varphi_{\epsilon} \parallel_{L^{2}} &\leq  \parallel \varphi_{\epsilon} \parallel_{L^{2}}. 
    \end{split}
\end{equation}

\textbf{Step 2:} General case.

Now we turn to the general case, where $\partial \Omega$ is not flat. Since by assumption $\partial \Omega$ is $C^{2},$ there exist finitely many balls $B(x_{i}, \epsilon_{i})$ and $C^{2}$ diffeomorphisms $g_{i}: U_{i} \subset \mathbf{R}^{d-1} \to \mathbf{R}^{d}$ such that
\begin{equation}
    g_{i}\left( U_{i} \right)= B(x_{i}, \epsilon_{i}) \bigcap \partial \Omega.
\end{equation}
For any $\delta < \epsilon_{*}, $ we can extend $g_{i}$ to a $C^{1}$ diffeomorphism $\overline{g}_{i}: U_{i} \times (-\delta, \delta) \to V_{i}^{\delta},$ where
\begin{equation}
V_{i}^{\delta}= \{ x+s\widehat{n}(x)| x \in  B(x_{i}, \epsilon_{i}) \bigcap \partial \Omega, s \in (-\delta, \delta) \}.
\end{equation}
We define $\overline{g}_{i}$ as
\begin{equation}
    \overline{g}_{i}(\underline{x}, s)= g_{i}(\underline{x})+s\nu(g_{i}(\underline{x})).
\end{equation}

Now define for any $\epsilon<\epsilon_{*}$ the function $\varphi_{\epsilon}^{i}:  U_{i} \times (-\epsilon, \epsilon)$ as in step 1, with $\alpha=\epsilon_{*}$. If $f$ is non-negative in a neighborhood of $\partial \Omega$, define $M \varphi_{\epsilon}^{i}$ as in step 1. 

Define the functions $\phi_{\epsilon}^{i}$ as 
\begin{equation}
    \phi_{\epsilon}^{i}= \varphi_{\epsilon}^{i}\circ \overline{g}_{i}^{-1}.
\end{equation}
 If $f$ is non-negative in a neighborhood of $\partial \Omega$, define the functions $M\phi_{\epsilon}^{i}$ as 
\begin{equation}
    M\phi_{\epsilon}^{i}= M\varphi_{\epsilon}^{i}\circ \overline{g}_{i}^{-1}.
\end{equation}

Lastly, take a partition of unity $q_{i}$ associated to $V_{i}^{\epsilon_{*}}.$ Define the extension $f_{\epsilon}$ as 
\begin{equation}
   f_{\epsilon}(x) =
   \begin{cases}
    f(x) \text{ if } x \in \Omega\\
    \sum_{i}(q_{i} \phi_{i}^{\epsilon})  \text{ if } x \in \bigcup V_{i}^{\delta}\\
    0 \text{ o.w.}
    \end{cases}
\end{equation}

 If $f$ is non-negative in a neighborhood of $\partial \Omega$, define the extension $Mf_{\epsilon}$ as 
\begin{equation}
   Mf_{\epsilon}(x) =
   \begin{cases}
    f(x) \text{ if } x \in \Omega\\
    \sum_{i}(q_{i} M\phi_{i}^{\epsilon})  \text{ if } x \in \bigcup V_{i}^{\delta}\\
    0 \text{ o.w.}
    \end{cases}
\end{equation}

It is easy to check that $f_{\epsilon}, Mf_{\epsilon}$ saitsfy the desired properties.

\end{proof}

We proceed to another lemma. 

\begin{lemma}
\label{lem:teclemma2}
Let $\nu \in H^{-s}(\mathbf{R}^{d})$. Assume that there exists a compact set $\Omega$ such that $\nu$ is nonpositive or nonnegative outside of $\Omega.$ Then there exists a compact set $\Omega_{2}$ which contains $\Omega,$ a constant $C,$ and a function $\phi \in H^{s}(\Omega_{2})$ such that 
\begin{equation}
    \int_{\Omega_{2}} \nu \phi \geq C \parallel \nu|_{\Omega_{2}} \parallel_{H^{-s}(\Omega_{2})},
\end{equation}
$\parallel \phi \parallel_{H^{s}}=1$, and $\phi$ is non-negative a neighborhood of $\partial \Omega_{2}.$ Furthermore, $C$ and $\Omega_{2}$ depends only on $\Omega$.
\end{lemma}

\begin{proof}

We assume WLOG that $\nu$ is positive outside of $\Omega$ and that $\partial \Omega \in C^{2}$. Otherwise we could find a compact set with $C^{2}$ boundary containing $\Omega$. Let 
\begin{equation}
    \Omega^{\epsilon}=\{x \in \mathbf{R}^{d} | d(x,\Omega) \leq \epsilon \}
\end{equation}
Take some $\epsilon>0$. Note that
\begin{equation}
    \parallel \nu|_{\Omega^{\epsilon}} \parallel_{H^{-s}(\Omega^{\epsilon})} < \infty,
\end{equation}
and hence, there exists some $\varphi \in H^{s}(\Omega^{\epsilon})$ such that
\begin{equation}
    \parallel \varphi \parallel_{H^{s} } =1
\end{equation}
and 
\begin{equation}
    \int_{\Omega^{\epsilon}} \nu \varphi \geq \frac{1}{2} \parallel \nu|_{\Omega^{\epsilon}} \parallel_{H^{-s}(\Omega^{\epsilon})}.
\end{equation}
Consider now $\overline{\varphi}=\varphi|_{\Omega}.$ By the extension lemma (Lemma \ref{extensionlemma}), there exists an extension $\widetilde{\varphi}$ of $\overline{\varphi}$ such that 
\begin{equation}
    \text{supp} \left( \widetilde{\varphi} \right) \subset \Omega^{\epsilon}
\end{equation}
and 
\begin{equation}
    \parallel \widetilde{\varphi} \parallel_{H^{s}} \leq C,
\end{equation}
where $C_{\epsilon}$ depends on $\Omega$ and $\epsilon$. Note that $\widetilde{\varphi}$ is $0$ in a neighborhood of $\partial_{\Omega^{\epsilon}}$.

Consider now 
\begin{equation}
      \phi=\max\{\varphi, \widetilde{\varphi}\}.
\end{equation}
Then since $\phi \geq \widetilde{\varphi},$ we know that $\phi$ is non-negative in a neighborhood of $\partial \Omega^{\epsilon}$. We also know that 
\begin{equation}
\begin{split}
      \varphi (x)&=\phi(x)  \text{ for } x \in \Omega, \\
      \varphi (x)&\leq \phi(x)  \text{ for } x \in \Omega^{\epsilon} \setminus \Omega, \\
\end{split}
\end{equation}
which implies
\begin{equation}
    \int_{\Omega^{\epsilon}} \nu \varphi \leq \int_{\Omega^{\epsilon}} \nu \phi,
\end{equation}
since $\nu$ is positive outside of $\Omega$. 

It can be easily shown that 
\begin{equation}
    \begin{split}
        \parallel \phi \parallel_{H^{s} } &\leq \parallel \varphi \parallel_{H^{s} } + \parallel \widetilde{\varphi} \parallel_{H^{s} }\\
        &\leq 1+C_{\epsilon}.
    \end{split}
\end{equation}
Taking $\widehat{\phi}=\frac{\phi}{ \parallel \phi \parallel_{H^{s} }}$, $C=\frac{1}{2(1+C_{\epsilon})}$ and $\Omega_{2}=\overline{\Omega}^{\epsilon}$ we obtain the result.
\end{proof}

We now prove Proposition \ref{localizingH-snorm}, restated here.

\begin{proposition}\label{localizingH-snorm2}
Let $\nu \in H^{-s}(\mathbf{R}^{d})$ and assume that there exists a compact set $\Omega$ such that $\nu$ is nonpositive or nonnegative outside of $\Omega$. Then there exists a compact set $\Omega_{1}$ which contains $\Omega,$ and a constant $C$ such that 
\begin{equation}
    \parallel \nu \parallel_{H^{-s}} \geq C \parallel \nu|_{\Omega_{1}} \parallel_{H^{-s}(\Omega_{1})}.
\end{equation}
Furthermore, $C$ and $\Omega_{1}$ depend only on $\Omega.$
\end{proposition} 

\begin{proof}
Again, WLOG we assume that $\nu$ is positive outside of $\Omega$. Take some fixed $\epsilon>0$. Then by Lemma \ref{lem:teclemma2} there exists a $\varphi \in H^{s}(\Omega^{\epsilon})$ such that 
\begin{itemize}
    \item[1.]
    \begin{equation}
        \parallel \varphi \parallel_{H^{s} }=1.
    \end{equation}
    \item[2.] $\varphi$ is positive in a neighborhood of $\partial \Omega^{\epsilon}$
    \item[3.]
    \begin{equation}
        \int_{\Omega^{\epsilon}} \nu \varphi \geq  C \parallel \nu|_{\Omega^{\epsilon}} \parallel_{H^{-s}(\Omega^{\epsilon})}, 
    \end{equation}
    where $C$ depends only on $\Omega$ and $\epsilon.$
\end{itemize}

By Lemma \ref{extensionlemma}, there exists an extension $\widehat{\varphi}$ of $\varphi$ such that 
\begin{itemize}
    \item[1.] The support of $\widehat{\varphi}$ is contained in $\overline{\Omega}^{1+\epsilon}$
    \item[2.] 
    \begin{equation}
        \parallel  \widehat{\varphi} \parallel_{\dot{H}^{s}}\leq C,
    \end{equation}
    where $C$ depends only on $\Omega$ and $\epsilon.$
    \item[3.] $\widehat{\varphi}$ is nonnegative in $\Omega^{1+\epsilon} \setminus \Omega^{\epsilon}.$
\end{itemize}
Since $\nu$ is positive outside of $\Omega$ and $\widehat{\varphi}$ is positive outside of $\Omega^{\epsilon},$ we have that 
\begin{equation}
    \int_{\Omega^{\epsilon}} \nu \varphi \leq \int_{\mathbf{R}^{d}} \nu \widehat{\varphi}.
\end{equation}
Finally, we have that 
\begin{equation}
    \begin{split}
        \parallel \nu \parallel_{H^{-s}} &\geq \frac{1}{\parallel \widehat{\varphi} \parallel_{\dot{H}^{s} } } \int_{\mathbf{R}^{d}} \nu \widehat{\varphi}\\
        &\geq \frac{1}{\parallel \widehat{\varphi} \parallel_{\dot{H}^{s} } } \int_{\Omega^{\epsilon}} \nu \varphi\\
        &\geq C  \parallel \nu|_{\Omega^{\epsilon}} \parallel_{H^{-s}(\Omega^{\epsilon})},
    \end{split}
\end{equation}
where $C$ depends only on $\Omega$ and $\epsilon.$

\end{proof}

\section{Appendix B: smearing for Riesz-type kernels}
\label{sect:appendixB}

In this section, we prove Proposition \ref{prop:energydist}. In order to do so, we need a fundamental result about regularization for Riesz-type kernels.

\begin{lemma}
\label{lem:RosSer}
Let $d \geq 1$ and $0 < s < \min \{ 1 , \frac{d}{2}\}$. Suppose that $X_{N} \in \mathbf{R}^{d \times N}$ is a pairwise distinct configuration and that $\mu \in \mathcal{P}(\mathbf{R}^{d}) \cap L^{\infty}(\mathbf{R}^{d})$. Then there exists a constant $C$ depending only on $s,d,g$ such that for every $\eta_{i} < \min \{\frac{1}{2}, \frac{r_{0}}{2}\}$ we have 
\begin{equation}
    \frac{1}{C} \left|\left| \frac{1}{N} \sum_{i=1}^{N} \delta_{x_{i}}^{(\eta_{i})} - \widetilde{\mu} \right|\right|_{\dot{H}^{- s - \frac{m}{2}}}^{2} \leq {\rm F}_{N}(X_{N}, \mu) + \frac{1}{N^{2}} \sum_{i=1}^{N} G_{\eta_{i}}(0) + \frac{C}{N}\sum_{i=1}^{N} \left( \|\mu\|_{L^{\infty}} \left[ \eta_{i}^{2s} + \eta_{i}^{2} \right] \right).
\end{equation}
\end{lemma}

In this notation, $\delta_{x}^{(\eta)}$ is the uniform probability measure on $\partial B(x, \eta) \subset \mathbf{R}^{d+m}$, $\widetilde{\mu} = \mu \delta_{\mathbf{R}^{d} \times \{0\}}$ is a probability measure on $\mathbf{R}^{d+m}$, and $G_{\eta} = G \ast \delta^{\eta}_{0}$. 

\begin{proof}
See \cite{nguyen2021mean}, Proposition 2.2.
\end{proof}

We now prove Proposition \ref{prop:energydist}, restated here.

\begin{proposition}
Let $\mu \in \mathcal{P}(\mathbf{R}^{d})$ have compact support $\Omega$ and let $\alpha > s$. Then there exists a constant $C$, which depends on $V,g, \Omega$, and $\alpha$ such that
\begin{equation}\label{eq:compineq}
    \| {\rm emp}_{N} -\mu \|_{{C}^{0,\alpha}_{*}} \leq N^{- \frac{\alpha}{d}} + C \left( {\rm F}_{N}(X_{N}, \mu) + C \|\mu\|_{L^{\infty}} N^{-\frac{2s}{d}} \right)^{\frac{1}{2}}.
\end{equation}
\end{proposition}

\begin{proof}
By Lemma \ref{L:locprob}, 
\begin{equation}
    \| {\rm emp}_{N}-\mu \|_{{C}^{0,\alpha}_{*}} \leq 2 \| ({\rm emp}_{N} - \mu)\mathbf{1}_{\Omega} \|_{{C}^{0,\alpha}_{*}}.
\end{equation}

Note that 
\begin{equation}
    \| ({\rm emp}_{N} - \mu)\mathbf{1}_{\Omega} \|_{{C}^{0,\alpha}_{*}} = \int_{\mathbf{R}^{d}} ({\rm emp}_{N} - \mu)\mathbf{1}_{\Omega} \phi ,
\end{equation}
for some $\phi$ with compact support and such that
\begin{equation}
    \|\phi\|_{{C}^{0,\alpha}} =1.
\end{equation}

Arguing as in \cite{nguyen2021mean}, let $\overline{\phi}: \mathbf{R}^{d+m} \to \mathbf{R}$ be such that $\overline{\phi}(x,0) = {\phi}(x) \forall x$ and
\begin{equation}
    \begin{split}
       \| \overline{\phi} \|_{H^{s + \frac{m}{2}}} &\leq C   \| {\phi} \|_{H^{s}}\\
       | \overline{\phi} |_{\dot{C}^{0,\alpha}} &\leq C   | {\phi} |_{\dot{C}^{0,\alpha}},
    \end{split}
\end{equation}
for some $C$ depending only on $m$ and $d$.

For example, take $\overline{\phi}(x,z) = \left( e^{|z| \Delta_{\mathbf{R}^{d}}} \phi \right)(x)$. We introduce the notation
\begin{equation}
    {\rm emp}_{N}^{\eta} = {\rm emp}_{N} \ast \delta^{(\eta)}_{0}.
\end{equation}
Now we can write
\begin{equation}
    \begin{split}
        \int_{\mathbf{R}^{d}} \phi \mathbf{1}_{\Omega} \, d ({\rm emp}_{N} - \mu) &\leq \left|   \int_{\mathbf{R}^{d+m}} \overline{\phi}\mathbf{1}_{\Omega  \times B(0,1)}({\rm emp}_{N}^{\eta} - \overline{\mu})  \right| + \left|  \int_{\mathbf{R}^{d+m}} \overline{\phi} \mathbf{1}_{\Omega  \times B(0,1)} \, d \left( {\rm emp}_{N} - {\rm emp}_{N}^{\eta} \right) \right|\\
        &\leq \| \overline{\phi} \mathbf{1}_{\Omega \times B(0,1)} \|_{H^{s + \frac{m}{2}}}\| \left( {\rm emp}_{N}^{\eta} - \overline{\mu} \right) \mathbf{1}_{\Omega \times B(0,1)} \|_{H^{-s - \frac{m}{2}}(\Omega \times B(0,1))} + \eta^{\alpha}.
    \end{split}
\end{equation}

Note that, since $\phi$ has compact support
\begin{equation}
    \begin{split}
        \| \overline{\phi} \mathbf{1}_{\Omega \times B(0,1)} \|_{H^{s + \frac{m}{2}}} & \leq  \| {\phi} \|_{H^{s}}\\
        & \leq C \| {\phi} \|_{{C}^{0,\alpha}},
    \end{split}
\end{equation}
where $C$ depends on $\Omega$.

On the other hand, using the localization inequality (Proposition \ref{localizingH-snorm}) and Lemma \ref{lem:RosSer},
\begin{equation}
    \begin{split}
        &\| \left( {\rm emp}_{N}^{\eta} - \overline{\mu} \right) \mathbf{1}_{\Omega \times B(0,1)} \|_{H^{-s - \frac{m}{2}}(\Omega \times B(0,1))}\\
        \leq & C \|  {\rm emp}_{N}^{\eta} - \overline{\mu} \|_{\dot{H}^{-s - \frac{m}{2}}}\\
         \leq & C\left( {\rm F}_{N}(X_{N}, \mu) + \frac{1}{N^{2}} \sum_{i=1}^{N} G_{\eta_{i}}(0) + \frac{C}{N}\sum_{i=1}^{N} \left( \|\mu\|_{L^{\infty}} \left[ \eta_{i}^{2s} + \eta_{i}^{2} \right] \right) \right)^{\frac{1}{2}},
    \end{split}
\end{equation}
where $C$ depends on $V,g, \Omega$, and $\alpha$.

Taking $\eta = N^{-\frac{1}{d}}$, we obtain equation \eqref{eq:compineq}.
\end{proof}

\section{Acknowledgements}

DPG acknowledges support by the German Research Foundation (DFG) via the research unit FOR 3013 “Vector- and tensor-valued surface PDEs” (grant no. NE2138/3-1). 

\bibliographystyle{plain}
\bibliography{bibliography.bib}

\end{document}